\documentclass[a4paper,10pt]{article}
\usepackage[utf8]{inputenc}
\usepackage{amsfonts}

\usepackage{float}
\usepackage{amsmath, amsthm, amssymb}
\usepackage{thmtools, thm-restate}

\usepackage{mathtools}

\usepackage{graphicx}

\usepackage{dsfont}

\usepackage{hyperref}
\usepackage{amsthm}
\hypersetup{
    colorlinks=true,                          
    linkcolor=blue, 
    citecolor=blue, 
    urlcolor=blue  }
    \usepackage{stackengine}

\newcommand{\cP}{\mathcal{P}}

\newcommand{\R}{\mathbb{R}}
\newcommand{\mub}{\overline{\mu}}
\newcommand{\Vb}{\overline{V}}
\newcommand{\Zb}{\overline{Z}}
\newcommand{\muh}{\widehat{\mu}}
\newcommand{\sigh}{\widehat{\sigma}}
\newcommand{\sigb}{\overline{\sigma}}
\newcommand{\lambh}{\widehat{\lambda}}
\newcommand{\lambb}{\overline{\lambda}}

\newcommand{\gb}{\overline{g}}
\newcommand{\bh}{\widehat{b}}
\newcommand{\bb}{\overline{b}}
\newcommand{\un}{\mathds{1}}

\usepackage{mathtools}
\newtheorem{theorem}{Theorem}[section]
\newtheorem{lemma}[theorem]{Lemma}
\newtheorem{remark}[theorem]{Remark}
\newtheorem{corollary}[theorem]{Corollary}

\usepackage[margin=65pt]{geometry}
\usepackage[T1]{fontenc}
\usepackage{cases} 
\usepackage{xcolor}
\usepackage{hyperref}
\usepackage[ruled]{algorithm2e}

\title{Influence of sampling on the convergence rates of greedy algorithms for parameter-dependent random variables}            

\author{\scshape{Mohamed-Raed Blel, Virginie Ehrlacher,Tony Leli\`evre}}
\date{\today}

\begin{document}

\maketitle

\begin{abstract}
The main focus of this article is to provide a mathematical study of the algorithm proposed in~\cite{boyaval2010variance} where the authors proposed a variance reduction technique for the computation of parameter-dependent expectations using a reduced basis paradigm. 
We study the effect of Monte-Carlo sampling on the theoretical properties of greedy algorithms. 
In particular, using concentration inequalities for the empirical measure in Wasserstein distance proved in~\cite{fournier2015rate}, we provide sufficient conditions on the number of samples used for the computation of empirical variances at 
each iteration of the greedy procedure to guarantee that the resulting method algorithm is a weak greedy algorithm with high probability. These theoretical results are not fully practical 
and we therefore propose a heuristic procedure to choose the number of Monte-Carlo samples at each iteration, inspired 
from this theoretical study, which provides satisfactory results on several numerical test cases. 
\end{abstract}

\section{Introduction}

The aim of this article is to provide a mathematical study of the algorithm proposed in~\cite{boyaval2010variance} 
where the authors proposed a variance reduction technique for the computation of parameter-dependent expectations using a reduced basis paradigm. 

\medskip

More precisely, the problematic we are considering here is the following: let us denote by $\mathcal P \subset \mathbb{R}^m$ a set of parameter values. In several applications, it is of significant interest 
to be able to rapidly compute the expectation of a random variable of the form
$f_\mu(Z)$ for a large numbers of values of the parameter $\mu\in \mathcal P$, where $Z$ is a random vector and where for all $\mu\in \mathcal P$, $f_\mu$ is a real-valued function. In practice,
such expectations may not be computable analytically 
and are approximated using empirical means involving a large number of random samples of the random vector $Z$. Variance reduction methods are commonly used in 
such contexts in order to reduce the computational cost of approximating these expectations by means of standard Monte-Carlo algorithms. 
Among these, control variates, which are chosen as approximations of the random variable $f_\mu(Z)$ the expectation of which can be easily computed, can yield to interesting gains in terms of computational cost, 
provided that the variance of the difference between 
$f_\mu(Z)$ and its approximation is small. The construction of efficient control variates for a given application is thus fundamental for the variance reduction technique to yield significant computational gains.

\medskip

In~\cite{boyaval2010variance}, the authors proposed a general algorithm in order to construct a control variate for $f_\mu(Z)$ using a reduced basis paradim. More precisely, the approximation of $f_\mu(Z)$ is constructed 
as a linear combination of 
$f_{\mu_1}(Z), \cdots, f_{\mu_n}(Z)$ for some small integer 
$n\in \mathbb{N}^*$ and well-chosen values $\mu_1, \cdots, \mu_n \in \mathcal P$ of the parameters. The choice of $n$ and of the values of the parameters stems from an iterative procedure, called a 
greedy algorithm, which consists at iteration $n\in \mathbb{N}$ to compute
$$
\mu_{n+1} \in \mathop{\rm argmax}_{\mu \in \mathcal P} \mathop{\inf}_{Z_n \in V_n} {\rm Var}\left[ f_\mu(Z) - Z_n \right],
$$
where $V_n:= {\rm Span}\left\{ f_{\mu_1}(Z), \cdots, f_{\mu_n}(Z)\right\}$. In the ideal (unpractical) case where variances can be exactly 
computed, the procedure boils down to a standard greedy algorithm in a Hilbert space~\cite{devore2013greedy}. It is now well-known~\cite{devore2013greedy} that such a greedy procedure provides 
a quasi-optimal set of parameters $\mu_1, \cdots, \mu_n$ in the sense that the error 
$$
\mathop{\sup}_{\mu \in \mathcal P} \mathop{\inf}_{Z_n \in V_n} {\rm Var}\left[ f_\mu(Z) - Z_n \right] = \mathop{\inf}_{Z_n \in V_n} {\rm Var}\left[ f_{\mu_{n+1}}(Z) - Z_n \right]
$$
is comparable to the so-called Kolmogorov $n$-width of the set $\left\{ f_\mu(Z), \mu \in \mathcal P\right\}$, defined by 
$$
\mathop{\sup}_{\mu \in \mathcal P} \mathop{\inf}_{
\begin{array}{c}
 W_n \mbox{ vectorial subspace}\\
 {\rm dim} W_n  = n \\
\end{array}
} \mathop{\inf}_{Z_n \in W_n} {\rm Var}\left[ f_\mu(Z) - Z_n \right].
$$
In other words, the subspace $V_n$ is a quasi-optimal subspace of dimension $n$ for the approximation of random variables $f_\mu(Z)$ for $\mu \in \mathcal P$ in an $L^2$ norm sense.   

However, in practice, variances cannot be computed exactly and have to be approximated by empirical means involving a finite number of samples of the random vector $Z$, which may be different from one iteration of the greedy algorithm to another. 
The main result of this article is to give theoretical lower bounds on the number of samples which have to be taken at each iteration of the greedy algorithm in order 
to guarantee that the resulting Monte-Carlo greedy algorithm enjoys quasi-optimality properties close to those of an ideal greedy algorithm with high probability.

\medskip

The mathematical analysis of algorithms which combine randomness and greedy procedures is a quite recent and active field of research among the model-order reduction community. 
Let us mention here a few works in this direction in which different settings than the one we focus on here 
are considered. In~\cite{cohen2020reduced}, the authors consider the effect of randomly sampling the set of parameters in order to define random trial sets at each iteration 
of the greedy algorithm and prove that the obtained procedure enjoys remarkable approximation properties which remain 
very close to the approximation properties of a greedy algorithm where minimization problems at each iteration are defined over the whole set of parameters. 
In~\cite{smetana2019randomized,smetana2019randomized2,balabanov2019randomized,balabanov2019randomized2}, the authors propose randomized residual-based error estimators 
for parametrized equations, with a view to using them for the acceleration of greedy algorithms for reduced basis techniques. Let us finally mention that
significant research efforts are devoted by many different groups to the improvement of randomized algorithms for 
Singular Value Decompositions~\cite{chakraborty2017comparative}, which plays a fundamental role for model-order reduction. 

\medskip

The outline of the article is the following. In Section~\ref{sec:motivation}, we motivate the interest of greedy algorithms for the construction of control variates for variance reduction methods and recall some results 
of~\cite{buffa2012priori, binev2011convergence, devore2013greedy} on the mathematical analysis of greedy algorithms in Hilbert spaces. In Section~\ref{sec:sampling}, 
we present the Monte-Carlo greedy algorithm, which is the main focus of this article, our main theoretical result and its proof. This theoretical result does not yield a fully practical algorithm. To alleviate this difficulty, we propose in 
Section~\ref{sec:resnum} a heuristic algorithm, inspired from the theoretical result, which provides satisfactory results on several test cases.

\section{Motivation: greedy algorithms for reduced bases and variance reduction}\label{sec:motivation}

\subsection{Motivation: reduced basis control variate}\label{sec:motiv}

The aim of this section is to present the motivation of our work, which aims at constructing control variates for reducing the variance of a Monte-Carlo estimator of the mean of parameter-dependent functions of random vectors. 

\medskip

Let us begin by introducing some notation. Let $d\in \mathbb{N}^*$, $(\Omega, \mathcal{F}, \mathbb{P})$ be a probability space and $Z$ a $\mathbb{R}^d$-valued random vector with associated probability measure $\nu$. For all $q\in \mathbb{N}^*$, 
we denote by 
$$
L^q_{\nu}(\mathbb{R}^d):= \left\{ f: \R^d \to \R, \quad \int_{\R^d} |f(x)|^q\,d\nu(x) < +\infty\right\}. 
$$
Let $\mathcal C(\mathbb{R}^d)$ denote the set of continuous real-valued functions defined on $\mathbb{R}^d$. 
Let $p\in \mathbb{N}^*$, $\mathcal{P} \subset \mathbb{R}^p$ be a set of parameter values, and for all $\mu\in \mathcal P$, let $f_\mu$ be an element of $\mathcal C(\mathbb{R}^d) \cap L^2_{\nu}(\mathbb{R}^d)$.

\medskip

For all $f,g\in \mathcal C(\mathbb{R}^d)$, any $M\in \mathbb{N}^*$ and any collection $\overline{Z}:= (Z_k)_{1\leq k \leq M}$ of random vectors of $\mathbb{R}^d$, we define the empirical averages:
\begin{align*}
\mathbb{E}_{\overline{Z}}(f)&:=\frac{1}{M}\sum^M_{k=1} f\left(Z_k\right),\\
{\rm Cov}_{\overline{Z}}(f,g)& :=\mathbb{E}_{\overline{Z}}(fg)  - \mathbb{E}_{\overline{Z}}(f) \mathbb{E}_{\overline{Z}}(g), \\
{\rm Var}_{\overline{Z}}(f)& :={\rm Cov}_{\overline{Z}}(f,f). \\
\end{align*}

The aim of our work is to propose and analyse from a mathematical point of view a numerical method in order to efficiently construct control variates to reduce the variance of a Monte-Carlo estimator of $\mathbb{E}\left[f_\mu(Z)\right]$ 
for all $\mu \in \mathcal{P}$ using a Reduced Basis paradigm~\cite{hesthaven2016certified,barrault2004empirical,quarteroni2015reduced,devore2014theoretical}, which was originally proposed in~\cite{boyaval2010variance}. 

\medskip

More precisely, let $M_{\rm small},M_{\rm ref} \in \mathbb{N}^*$ and assume that $M_{\rm ref}\gg M_{\rm small}$. Let $\overline{Z}^{\rm ref}:= \left(Z^{\rm ref}_k\right)_{1\leq k \leq M_{\rm ref}}$ and 
$\overline{Z}^{\rm small}:=\left(Z^{\rm small}_k\right)_{1\leq k \leq M_{\rm small}}$ 
be two independent collections of iid random vectors distributed according to the law of $Z$ and \bfseries independent of $Z$\normalfont. 
\medskip

Let us assume that we have selected $N$ values of parameters $(\mu_1,\mu_2,...,\mu_N) \in \mathcal{P}^N$ for some $N\in \mathbb{N}^*$ and assume that the empirical means
$(\mathbb{E}_{\overline{Z}^{\rm ref}}(f_{\mu_i}))_{1\leq i\leq N}$ have been computed in an offline phase. 

In an online phase, for all $\mu \in \mathcal{P}$, we can build an approximation of $\mathbb{E}\left[f_\mu(Z)\right]$, using a control variate which reads as $\overline{f}_\mu(Z)$
for some function $\overline{f}_\mu: \mathbb{R}^d \to \mathbb{R}$:
            \begin{equation}\label{eq:controlvar}
\mathbb{E}\left[f_\mu(Z)\right] \approx \mathbb{E}_{\overline{Z}^{\rm ref}}(\overline{f}_\mu)+\mathbb{E}_{\overline{Z}^{\rm small}}\left(f_\mu-\overline{f}_\mu\right).
            \end{equation}
            
\begin{remark}\label{rem:rem1}
Let us point out that the statistical error between $\mathbb{E}_{\overline{Z}^{\rm ref}}(\overline{f}_\mu)$ and $\mathbb{E}\left[ \overline{f}_\mu(Z)\right]$ is close to 
$$
\sqrt{\frac{{\rm Var}\left[ \overline{f}_\mu(Z)\right]}{M_{\rm ref}}},
$$
whereas the error between $\mathbb{E}_{\overline{Z}^{\rm small}}\left(f_\mu -\overline{f}_\mu\right)$ and $\mathbb{E}\left[ (f_\mu - \overline{f}_\mu)(Z) \right]$ 
is of the order of 
$$\sqrt{\frac{{\rm Var}\left[ \left(f_\mu - \overline{f}_\mu\right)(Z)\right]}{M_{\rm small}}}.
$$
The aim of the Monte-Carlo greedy algorithm studied in this article is to give an approximation of 
$\mathbb{E}\left[f_\mu(Z)\right]$ with an error close to $\sqrt{\frac{{\rm Var}\left[f_\mu(Z)\right]}{M_{\rm ref}}}$ within a much smaller computational 
time than the one required by the computation of $\mathbb{E}_{\overline{Z}^{\rm ref}}(f_\mu)$.
\end{remark}

In the method studied here, the control variate function $\overline{f}_\mu$ is constructed as follows:
$$
\overline{f}_\mu =\sum_{i=1}^N \lambda^\mu_{i}f_{\mu_i}
$$
where $\lambda^\mu:=\left(\lambda^{\mu}_i\right)_{1\leq i \leq N}\in \R^N$ is a solution of the linear system 
\begin{equation}\label{eq:linsys}
A\lambda^\mu = b^{\mu}
\end{equation}
where $A:=\left( A_{ij} \right)_{1\leq i,j \leq N} \in \R^{N\times N}$ and $b^{\mu}:= \left( b^{\mu}_i\right)_{1\leq i \leq N} \in \R^N$ are defined as follows: for all $1\leq i,j \leq N$, 
 \begin{equation}\label{eq:defMMS}
A_{ij} = \mathrm{Cov}_{\overline{Z}^{\rm small}}(f_{\mu_i}, f_{\mu_j}) \quad \mbox{ and } \quad b^{\mu}_{i} = \mathrm{Cov}_{\overline{Z}^{\rm small}}(f_\mu, f_{\mu_i}).
\end{equation}
Equivalently, the vector $\lambda^\mu$ is a solution of the minimization problem
$$
\lambda^\mu \in \mathop{\rm argmin}_{\lambda:=(\lambda_i)_{1\leq i \leq N}\in \R^N} {\rm Var}_{\overline{Z}^{\rm small}}\left( f_\mu - \sum_{i=1}^N \lambda_i f_{\mu_i}\right).
$$
Let us point out that $\lambda^\mu$ is a random vector which can be written as a deterministic function of $\overline{Z}^{\rm small}$. In other words, $\lambda^\mu$ is measurable with respect to $\overline{Z}^{\rm small}$. 
Remarking that $ \mathbb{E}_{\overline{Z}^{\rm ref}}(\overline{f}_\mu)= \sum_{i=1}^N \lambda_i^\mu \mathbb{E}_{\overline{Z}^{\rm ref}}(f_{\mu_i})$, the computation of the approximation (\ref{eq:controlvar}) of 
$\mathbb{E}\left[ f_\mu(Z)\right]$ thus 
requires the following steps: 
\begin{itemize}
 \item \bfseries offline phase: \normalfont Compute $(\mathbb{E}_{\overline{Z}^{\rm ref}}(f_{\mu_i}))_{1\leq i\leq N}$ ($N$ empirical means with
 $M_{\rm ref}$ samples), $(\mathbb{E}_{\overline{Z}^{\rm small}}(f_{\mu_i}))_{1\leq i\leq N}$ 
 ($N$ empirical means with 
 $M_{\rm small}$ samples) and the matrix $A$ ($N^2$ empirical covariances with $M_{\rm small}$ samples). 
 \item \bfseries online phase: \normalfont For all $\mu \in \mathcal P$, compute $b^\mu$ ($N$ empirical covariances with $M_{\rm small}$ samples) and solve the linear system (\ref{eq:linsys}) to obtain $\lambda^\mu$.
 Then, compute the approximation (\ref{eq:controlvar}) 
 of $\mathbb{E}\left[f_\mu(Z)\right]$ as
 \begin{equation}\label{eq:approx}
 \mathbb{E}\left[f_\mu(Z)\right] \approx  \sum_{i=1}^N \lambda_i^\mu \mathbb{E}_{\overline{Z}^{\rm ref}}(f_{\mu_i}) + \mathbb{E}_{\overline{Z}^{\rm small}}\left(f_\mu\right)-\sum_{i=1}^N \lambda_i^\mu \mathbb{E}_{\overline{Z}^{\rm small}}
 \left(f_{\mu_i}\right),
 \end{equation}
\end{itemize}
which requires $\mathcal{O}(N)$ elementary operations and the computation of one empirical mean with $M_{\rm small}$ samples.
  
\medskip
 
Naturally, the approximation of $\mathbb{E}\left[f_\mu(Z)\right]$ given by (\ref{eq:controlvar}) can be interesting from a computational point of view in terms of variance reduction only if 
$\mbox{\rm Var}\left[ f_\mu(Z) - \overline{f}_\mu(Z) \right]$ is much smaller than $\mbox{\rm Var}\left[ f_\mu(Z)\right]$. The following question thus naturally arises: how can the set of parameters
$(\mu_1,\mu_2,..,\mu_N) \in \mathcal{P}^N$ be chosen in the offline phase in order to ensure that $\mbox{\rm Var}\left[ f_\mu(Z) - \overline{f}_\mu(Z) \right]$ is as small as possible for any value of $\mu \in \mathcal P$? 

\medskip

Greedy algorithms stand as the state-of-the-art technique to construct such sets of snapshot parameters, enjoy very nice mathematical properties and are the backbone of the method proposed in~\cite{boyaval2010variance}
which we wish to analyze here. We present this family of algorithms and related existing theoretical convergence results in the next section.

\subsection{Greedy algorithms for reduced basis}

Let us recall here the results of~\cite{buffa2012priori, binev2011convergence, devore2013greedy} on the convergence rates of greedy algorithms for reduced bases, adapted to our context. Let us define
$$
L^2_{\nu,0}(\R^d):=\left\{ g \in L^2_\nu(\R^d), \; \int_{\R^d} g \,d\nu = 0 \right\}. 
$$
It holds that $L^2_{\nu,0}(\R^d)$ is a Hilbert space, equipped with the scalar product $\langle \cdot, \cdot \rangle$ defined by 
$$
\forall g_1,g_2\in L^2_{\nu,0}(\R^d), \quad \langle g_1, g_2\rangle = \int_{\R^d} g_1g_2\,d\nu = {\rm Cov}\left[ g_1(Z), g_2(Z)\right].
$$
The associated norm is denoted by $\|\cdot\|$ and is given by
$$
\forall g \in L^2_{\nu,0}(\R^d), \quad \|g\| = \left(\int_{\R^d} |g|^2\,d\nu\right)^{1/2} = \sqrt{{\rm Var}\left[g(Z)\right]}.
$$

\medskip

For all $\mu \in \mathcal P$, let us define 
\begin{equation}\label{eq:defgmu}
g_\mu:= f_\mu - \mathbb{E}\left[ f_\mu(Z)\right]
\end{equation}
and let us denote by 
\begin{equation}\label{eq:defM}
\mathcal M:= \left\{ g_\mu, \; \mu \in \mathcal P\right\}
\end{equation}
so that $\mathcal M \subset L^2_{\nu,0}(\R^d)$. Let us assume that $\mathcal M$ is a compact subset of $L^2_{\nu,0}(\R^d)$. For all $n\in \mathbb{N}^*$, we introduce the Kolmogorov $n$-width of the set $\mathcal M$ in $L^2_{\nu,0}(\R^d)$, defined by
\begin{align*}
d_n(\mathcal M) & := \mathop{\inf}_{
\begin{array}{c}
 V_n \subset L^2_{\nu,0}(\R^d) \; {\rm subspace},\\
 {\rm dim}\; V_n = n\\
\end{array}
}
\mathop{\sup}_{\mu \in \mathcal P}
\mathop{\inf}_{g_n\in V_n} \sqrt{{\rm Var}\left[ g_\mu(Z) - g_n(Z)\right]}\\
& = \mathop{\inf}_{
\begin{array}{c}
 V_n \subset L^2_{\nu,0}(\R^d) \; {\rm subspace},\\
 {\rm dim}\; V_n = n\\
\end{array}
}
\mathop{\sup}_{\mu \in \mathcal P}
\mathop{\inf}_{g_n\in V_n} \left\| g_\mu - g_n\right\|.\\
\end{align*}

Let $0< \gamma < 1$ and consider the following \itshape weak greedy algorithm \normalfont with parameter $\gamma$.

\vspace{0.5cm}

\noindent

\bfseries \underline{Weak-Greedy Algorithm}: \normalfont

\medskip

\bfseries Initialization: \normalfont Find $\mu_1\in \mathcal P$ such that
\begin{equation}\label{eq:init}
\left\|g_{\mu_1}\right\|^2 \geq \gamma^2 \mathop{\max}_{\mu \in \mathcal P} \left\|g_{\mu}\right\|^2.
\end{equation}
Set $V_1:= \mbox{\rm Span}\{ g_{\mu_1}\}$ and set $n=2$.
 
\bfseries Iteration $n\geq 2$: \normalfont Find $\mu_n\in \mathcal P$ such that
\begin{equation}\label{eq:iter}
\mathop{\inf}_{(\lambda_i)_{1\leq i \leq n-1} \in \R^{n-1}} \left\|g_{\mu_n} - \sum_{i=1}^{n-1} \lambda_i g_{\mu_i}\right\|^2 \geq \gamma^2 \mathop{\max}_{\mu \in \mathcal P} \mathop{\inf}_{(\lambda_i)_{1\leq i \leq n-1} \in 
\R^{n-1}} \left\|g_\mu- \sum_{i=1}^{n-1} \lambda_i g_{\mu_i}\right\|^2,
\end{equation}
Set $V_n:= V_{n-1} + \mbox{\rm Span}\{ g_{\mu_n}\} = \mbox{\rm Span}\{ g_{\mu_1},\cdots, g_{\mu_n}\} $. 

\vspace{0.5cm}

For all $n\in \mathbb{N}^*$, the error associated with the $n$-dimensional subspace $V_n$ given by the weak greedy algorithm is defined by
$$
\sigma_n(\mathcal M):= \mathop{\max}_{\mu \in \mathcal P} \mathop{\inf}_{(\lambda_i)_{1\leq i \leq n} \in \R^{n}}  \left\|g_\mu - \sum_{i=1}^{n} \lambda_i g_{\mu_i}\right\|.
$$

The following result is then a direct corollary of the results proved in~\cite[Corollary~3.3]{devore2013greedy}.
\begin{theorem}\label{cor:weakgreed}
 For all $n\in \mathbb{N}^*$, $\displaystyle \sigma_n\left(\mathcal{M}\right)\leq \sqrt{2} \gamma^{-1} \mathop{\min}_{0\leq m <n} \left(d_m(\mathcal M)\right)^{\frac{n-m}{n}}$. In particular, for all $n\in \mathbb{N}^*$, 
 $\sigma_{2n}\left(\mathcal{M}\right) \leq \sqrt{2} \gamma^{-1} \sqrt{d_n(\mathcal{M})}$.
\end{theorem}

This result indicates that the weak greedy algorithm provides a practical way to construct a quasi-optimal sequence $(V_n)_{n\in \mathbb{N}^*}$ of finite dimensional subspaces of $L^2_{\nu, 0}(\R^d)$. 

\medskip

Of course, the weak greedy algorithm introduced above cannot be implemented in practice since it requires at the $n^{\rm th}$ iteration of the algorithm 
the computation of the exact variances of $g_{\mu}(Z) - \sum_{i=1}^{n-1} \lambda_i g_{\mu_i}(Z)$ for $\mu,\mu_1,\cdots,\mu_{n-1}\in \mathcal{P}$ and $\lambda_1,\cdots, \lambda_{n-1}\in \R$, which is out of reach in our context. 
In practice, these quantities have to be approximated by Monte-Carlo estimators involving a finite number of samples of the random vector $Z$. The resulting greedy algorithm with Monte Carlo sampling 
is presented in Section~\ref{sec:sampling}. The mathematical analysis of this algorithm is the main purpose of the present article. 

\medskip

For the sake of simplicity, in the rest of the article, we assume that for all $n\in \mathbb{N}^*$, $d_n(\mathcal M) >0$.

\section{Greedy algorithm with Monte-Carlo sampling}\label{sec:sampling}

\subsection{Presentation of the algorithm}\label{sec:presentation}

Let us begin by presenting the greedy algorithm with Monte Carlo sampling. 

\medskip

Let $(M_n)_{n\in \mathbb{N}^*}$ be a sequence of integers,  which represents the number of samples used at iteration $n$.
For all $n\in \mathbb{N}^*$, let $\Zb^n:=(Z_k^n)_{1\leq k \leq M_n}$ be a collection of random vectors such that 
$(Z_k^n)_{n\geq 1, \; 1\leq k \leq M_n}$ are independent and identically distributed according to the law of $Z$, and independent of $Z$. 
Let $\Zb^{1:n}:= \left( \Zb^m\right)_{1\leq m \leq n}$ and $\Zb^{1:\infty}:= \left( \Zb^n\right)_{n\in \mathbb{N}^*}$. 

\medskip

For any random functions $g_1$, $g_2$ with values in $L^2_{\nu,0}(\R^d)$, we define
$$
 \langle g_1, g_2\rangle_{\overline{Z}^{1:\infty}} := {\rm Cov}\left[ g_1(Z), g_2(Z)\left| \overline{Z}^{1:\infty} \right. \right] \quad \mbox{ and } \quad \|g_1\|_{\overline{Z}^{1:\infty}} :=\sqrt{{\rm Var}\left[g_1(Z)\left| \overline{Z}^{1:\infty} \right.\right]}.
$$

Let us make here an important remark. Since $\Zb^{1:\infty}$ is a collection of random vectors which are all independent of $Z$, it holds that, for all $f,g \in L^2_{\nu,0}(\R^d)$, almost surely, 
\begin{align*}
\langle f, g\rangle_{\overline{Z}^{1:\infty}} & =  {\rm Cov}\left[ f(Z), g(Z)\left|\Zb^{1:\infty}\right.\right] = {\rm Cov}\left[ f(Z), g(Z)\right] = \langle f,g\rangle,\\
\|g\|^2_{\overline{Z}^{1:\infty}} & = {\rm Var}\left[g(Z)\left|\Zb^{1:\infty}\right.\right] = {\rm Var}[g(Z)] = \|g\|^2.\\  
\end{align*}
Hence, almost surely, $\langle \cdot, \cdot \rangle_{\Zb^{1:\infty}}$ defines a scalar product on $L^2_{\nu,0}(\R^d)$, which is a Hilbert space when equipped with this scalar product, 
and $\|\cdot\|_{\Zb^{1:\infty}}$ is the associated norm. 

\bigskip

The greedy algorithm with Monte-Carlo sampling reads as follows:

\vspace{0.5cm}

\noindent

\bfseries \underline{MC-Greedy Algorithm}: \normalfont

\medskip

\bfseries Initialization: \normalfont Find $\mub_1\in \mathcal P$ such that, almost surely,
\begin{equation}\label{eq:init}
\mub_1 \in \mathop{\rm argmax}_{\mu \in \mathcal P} {\rm Var}_{\overline{Z}^1}\left( g_{\mu}\right) \quad \mbox{ and } \quad g_{\mub_1} \neq 0. 
\end{equation}
Set $\Vb_1:= \mbox{\rm Span}\{ g_{\mub_1}\} $ and set $n=2$.
 
\bfseries Iteration $n\geq 2$: \normalfont Find $\mub_n\in \mathcal P$ such that, almost surely,
\begin{equation}\label{eq:iter}
\mub_n \in \mathop{\rm argmax}_{\mu \in \mathcal P} \mathop{\inf}_{(\lambda_i)_{1\leq i \leq n-1} \in \R^{n-1}}  {\rm Var}_{\overline{Z}^n}\left(g_\mu - \sum_{i=1}^{n-1} \lambda_i g_{\mub_i}\right) 
\quad \mbox{ and } \quad g_{\mub_n} \notin \Vb_{n-1}. 
\end{equation}
Set $\Vb_n:= \Vb_{n-1} + \mbox{\rm Span}\{ g_{\mub_n}\} = \mbox{\rm Span}\{ g_{\mub_1},\cdots, g_{\mub_n}\}$. 

\vspace{0.5cm}
Naturally, for all $n\in \mathbb{N}^*$, the parameter $\mub_n$ and thus the finite-dimensional space $\Vb_n$ are $\Zb^{1:n}$-measurable. 

\medskip

Let us first prove an auxiliary lemma. 

\begin{lemma}\label{lem:def}
 Almost surely, all the iterations of the MC-Greedy Algorithm are well-defined, in the sense that, for all $n\in \mathbb{N}^*$, there always exists at least one element $\overline{\mu}_n \in \mathcal P$ such that 
 (\ref{eq:init}) (when $n=1$) or (\ref{eq:iter}) (when $n\ge 2$) is satisfied. 
\end{lemma}

\begin{proof}[Proof of Lemma~\ref{lem:def}]
Let us first consider the initialization step corresponding to $n=1$. Two situations may a priori occur : 
either $\displaystyle \mathop{\max}_{\mu \in \mathcal P} {\rm Var}_{\overline{Z}^1}\left( g_{\mu}\right) >0$ or 
$\displaystyle \mathop{\max}_{\mu \in \mathcal P} {\rm Var}_{\overline{Z}^1}\left( g_{\mu}\right) = 0$. In the first case, choosing $\displaystyle 
\mub_1 \in \mathop{\rm argmax}_{\mu \in \mathcal P} {\rm Var}_{\overline{Z}^1}\left( g_{\mu}\right)$ is sufficient to guarantee that $g_{\mub_1} \neq 0$. Indeed, 
since $\mathcal M \subset \mathcal C(\mathbb{R}^d)$ (remember that $f_\mu$ is continuous for all $\mu \in \mathcal P$, and hence so is $g_\mu$), the fact that 
${\rm Var}_{\overline{Z}^1}\left( g_{\mub_1}\right) >0$ necessarily implies that ${\rm Var}\left[ g_{\mub_1}(Z)\left| \overline{Z}^{1:\infty}\right. \right] >0$ almost surely. 
Since  $\overline{Z}^{1:\infty}$ is independent of $Z$ and $\mub_1$ is a $\overline{Z}^{1:\infty}$ measurable random variable, this implies that almost surely $g_{\mub_1}\neq 0$.

\medskip

In the second case, it then holds that $\displaystyle {\rm Var}_{\overline{Z}^1}\left( g_{\mu}\right) = 0$ for all $\mu \in \mathcal P$. Then, 
the assumption $d_1(\mathcal M)>0$ implies that, almost surely, there exists at least one element $\mub_1 \in \mathcal P$ such that 
$g_{\mub_1} \neq 0$. In addition, $\mub_1 \in \mathop{\rm argmax}_{\mu \in \mathcal P} {\rm Var}_{\overline{Z}^1}\left( g_{\mu}\right)$.

\medskip

Using similar arguments and the fact that $d_n(\mathcal M)>0$ for all $n\in \mathbb{N}^*$, it is easy to see that, almost surely, all the iterations of the MC-Greedy algorithm are well-defined, 
in particular for $n\geq 2$. 
\end{proof}

\begin{remark}
We stress on the fact that the practical implementation of the MC-greedy algorithm does not require the knowledge of the value of $\mathbb{E}[f_\mu(Z)]$, even if $g_\mu= f_\mu - \mathbb{E}[f_\mu(Z)]$ 
for all $\mu \in \mathcal P$. Indeed, it holds that for all $g\in \mathcal{C}(\R^d)$, all $n\in\mathbb{N}^*$ and all $C\in \R$, ${\rm Var}_{\overline{Z}^n}\left( g\right)  = {\rm Var}_{\overline{Z}^n}\left( g + C\right)$. 
Thus, for all $\mu \in \mathcal{P}$, $n\in \mathbb{N}^*$ and $\lambda:=\left( \lambda_i\right)_{1\leq i \leq n-1} \in \R^{n-1}$, 
$$
{\rm Var}_{\overline{Z}^1}\left( g_{\mu}\right) = {\rm Var}_{\overline{Z}^1}\left( f_{\mu}\right) \quad \mbox{ and } \quad {\rm Var}_{\overline{Z}^n}\left(g_\mu - \sum_{i=1}^{n-1} \lambda_i g_{\mub_i}\right) = {\rm Var}_{\overline{Z}^n}\left(f_\mu - \sum_{i=1}^{n-1} \lambda_i f_{\mub_i}\right).
$$
Thus, the MC-greedy algorithm naturally makes sense with a view to the construction of a reduced basis control variate for variance reduction as explained in Section~\ref{sec:motiv}.
%
%
\end{remark}

\begin{remark}
In practice, a discrete subset $\mathcal P_{\rm trial} \subset \mathcal P$ has to be introduced. The optimization problems (\ref{eq:init}) and (\ref{eq:iter}) have to be replaced respectively by
$$
\mub_1 \in \mathop{\rm argmax}_{\mu \in \mathcal P_{\rm trial}} {\rm Var}_{\overline{Z}^1}\left( g_{\mu}\right) \quad \mbox{ and } \quad g_{\mub_1} \neq 0,
$$
and 
$$
\mub_n \in \mathop{\rm argmax}_{\mu \in \mathcal P_{\rm trial}} \mathop{\inf}_{(\lambda_i)_{1\leq i \leq n-1} \in \R^{n-1}}  {\rm Var}_{\overline{Z}^n}\left(g_\mu - \sum_{i=1}^{n-1} \lambda_i g_{\mub_i}\right) \quad \mbox{ and } \quad g_{\mub_n} \notin \Vb_{n-1}. 
$$
The influence of the choice of the set $\mathcal P_{\rm trial}$ on the mathematical properties of the MC-greedy algorithm is an important question which we do not address in our analysis for the sake of simplicity. 
For related discussion, we refer the reader 
to the work~\cite{cohen2020reduced}, where the authors study the mathematical properties of a greedy algorithm where the set $\mathcal P_{\rm trial}$ depends on the iteration $n$ of the greedy algorithm and is randomly chosen according to 
appropriate probability distributions defined on the set of parameters $\mathcal P$. 
 \end{remark}

For all $n\in \mathbb{N}^*$, we also define
\begin{align} \label{eq:sigh}
 \sigh_{n-1}(\mathcal M)& :=\mathop{\max}_{\mu \in \mathcal P} \mathop{\inf}_{(\lambda_i)_{1\leq i \leq n-1} \in \R^{n-1}}  \sqrt{{\rm Var}\left[\left. g_\mu(Z) - \sum_{i=1}^{n-1} \lambda_i g_{\mub_i}(Z)\right| \Zb^{1:\infty}\right]},\\ \label{eq:sigb}
 \sigb_{n-1}(\mathcal M)& := \mathop{\inf}_{(\lambda_i)_{1\leq i \leq n-1} \in \R^{n-1}}  \sqrt{{\rm Var}\left[\left. g_{\mub_n}(Z) - \sum_{i=1}^{n-1} \lambda_i g_{\mub_i}(Z)\right| \Zb^{1:\infty}\right]}\\ \nonumber
\end{align}

%
%

Let us point out here that $\sigh_{n-1}(\mathcal M)$ is a random variable which is measurable with respect to $\overline{Z}^{1:(n-1)}$ whereas $\mub_n$ and $\sigb_{n-1}(\mathcal M)$ are measurable with respect 
to $\overline{Z}^{1:n}$.

\subsection{Main theoretical result}

The aim of this section is to study the effect of Monte-Carlo sampling on the convergence of such a greedy algorithm. We consider here the probability space $(\Omega, \mathcal A(\overline{Z}^{1:\infty}), \mathbb{P})$ the probabilty 
space where $\mathcal A(\overline{Z}^{1:\infty})$ denotes the set of events that are measurable with respect to $\overline{Z}^{1:\infty}$. We prove, under appropriate assumptions on the probability density $\nu$
and on the set of functions $\mathcal M = \left\{g_\mu, \mu\in \mathcal P\right\}$, that for all $0< \gamma < 1$, there exist explicit conditions on the sequence $(M_n)_{n\in\mathbb{N}^*}$ so that, with high probability, the MC-greedy algorithm is actually 
a weak greedy algorithm with parameter $\gamma$. More precisely, under this set of assumptions, we prove that, with high probability, it holds that for all $n\in \mathbb{N}^*$, 
$$
\sigb_{n-1}(\mathcal M) \geq \gamma  \sigh_{n-1}(\mathcal M).
$$

\medskip

Let us now present the set of assumptions we make on $\nu$ and on the set $\mathcal M = \left\{ g_\mu, \; \mu \in \mathcal P\right\}$ for our main result to hold. 

\medskip

From now on, we make the following assumption on the probability distribution $\nu$.

\medskip

\bfseries Assumption (A): \normalfont The probability law $\nu$ is such that there exist $\alpha >1$ and $\beta >0$ such that 
$$
\int_{\R} e^{\beta |x|^\alpha}\,d\nu(x) < +\infty.
$$

Let us denote by $\mathcal L$ the set of Lipschitz functions of $\R^d$ and for all $f\in \mathcal L$, 
let us denote by $\|f\|_{\mathcal L}$ its Lipschitz constant. In the sequel, we denote by $\phi:\R_+^* \to \R_+^*$ the function defined by
\begin{equation}\label{eq:defphi}
\forall \kappa\in \R_+^*, \quad \phi(\kappa):= \left\{ 
\begin{array}{ll}
 \kappa^2 \un_{\kappa \leq 1} + \kappa^\alpha \un_{\kappa >1} & \quad \mbox{ if }d=1,\\
 (\kappa/\log(2 + 1/\kappa)^2) \un_{\kappa \leq 1}   + \kappa^\alpha \un_{\kappa >1}& \quad \mbox{ if }d=2,\\
  \kappa^d \un_{\kappa \leq 1}   + \kappa^\alpha \un_{\kappa >1} & \quad \mbox{ if }d\geq 3.\\
\end{array}
\right.
\end{equation}

A key ingredient in our analysis is the use 
of concentration inequalities in the Wasserstein-1 distance between a probability distribution and its empirical measure proved in~\cite{bolley2007quantitative, fournier2015rate}. 
Let us recall here a direct corollary of Theorem~2 of~\cite{fournier2015rate}, which is the backbone of our analysis.
\begin{corollary}\label{cor:concentration}
Let us assume that $\nu$ satisfies assumption (A). Then, there exist positive constants $c,C$ depending only on $\nu$, $d$, $\alpha$ and $\beta$, such that, 
for all $M\in \mathbb{N}^*$, all $\overline{Z}:=(Z_k)_{1\leq k \leq M}$ iid random vectors distributed according to $\nu$ and all $\kappa>0$, it holds that
$$
\mathbb{P}\left[ \mathcal{T}_1\left(\overline{Z}\right) \geq \kappa\right] \leq C e^{-cM\phi(\kappa)},
$$
where 
$$
\mathcal{T}_1\left( \overline{Z}\right):= \mathop{\sup}_{f \in \mathcal L; \|f\|_{\mathcal L} \leq 1} \left| \mathbb{E}[f(Z)] - \mathbb{E}_{\overline{Z}}(f)\right|.
$$
\end{corollary}

\begin{remark}
We would like to mention here that other concentration inequalities are stated in Theorem~2 of~\cite{fournier2015rate} under 
different sets of assumptions than (A) on the probability law $\nu$. In particular, weaker concentration inequalities may be obtained when $\nu$ only has some finite polynomial moments. 
Our analysis can then be easily adapted to these different settings but we restrict ourselves here to a framework where $\nu$ satisfies Assumption (A) for the sake of clarity. 
\end{remark}

\medskip

We finally make the following set of assumptions on $\mathcal M$ defined in (\ref{eq:defM}).

\medskip

\bfseries Assumption (B): \normalfont The set $\mathcal M$ satisfies the four conditions:
\begin{itemize}
 \item[(B1)] $\mathcal M$ is a compact subset of $L^2_{\nu, 0}(\R^d)$ and let $K_2:= \sup_{\mu \in \mathcal P} \|g_\mu\| < \infty$;
 \item[(B2)] $\mathcal M \subset \mathcal L$ and $K_{\mathcal L}:= \sup_{\mu \in \mathcal P} \|g_\mu\|_{\mathcal L} < +\infty$; 
 \item[(B3)] $\mathcal M \subset L^\infty(\R^d)$ and $K_{\infty}:= \sup_{\mu \in \mathcal P} \|g_\mu\|_{L^\infty} < +\infty$;
  \item[(B4)] for all $n\in \mathbb{N}^*$, $d_n(\mathcal M)>0$.
\end{itemize}

\medskip

Before presenting our main result, we need to introduce some additional notation. Using Lemma~\ref{lem:def}, we can almost surely define the sequence $(\gb_n)_{n\in\mathbb{N}^*}$ as the 
orthonormal family of $L^2_{\nu,0}(\R^d)$ obtained by a Gram-Schmidt orthonormalization procedure (for~$\|\cdot\|_{\overline{Z}^{1:\infty}}$) 
from the family $(g_{\mub_n})_{n\in\mathbb{N}^*}$. 
More precisely, we define
$$
\gb_1:= \frac{g_{\mub_1}}{\sqrt{{\rm Var}\left[g_{\mub_1}(Z)\left| \overline{Z}^{1:\infty}\right.\right] }}.
$$
Moreover, for all $n\geq 2$, et $\lambb^n:=\left(\lambb^n_i\right)_{1\leq i \leq n-1} \in \mathbb{R}^{n-1}$ be a solution to the minimization problem
$$
\lambb^n  \in \mathop{{\rm argmin}}_{\lambda:=\left(\lambda_i\right)_{1\leq i \leq n-1}\in \R^{n-1}} {\rm Var}\left[ \left. g_{\mub_n}(Z) - \sum_{i=1}^{n-1} \lambda_i \gb_i(Z) \right| \Zb^{1:\infty}\right].
$$
Then it holds that
$$
\gb_n := \frac{g_{\mub_n} - \sum_{i=1}^{n-1} \lambb^n_i \gb_i}{ \sqrt{{\rm Var}\left[ \left. g_{\mub_n}(Z) - \sum_{i=1}^{n-1} \lambb^n_i \gb_i(Z) \right| \Zb^{1:\infty}\right]}}.
$$
As a consequence, it always holds that $\Vb_n = {\rm Span}\left\{ g_{\mub_1}, \cdots, g_{\mub_n}\right\} =  {\rm Span}\left\{ \gb_1, \cdots, \gb_n\right\}$. Moreover, $\gb_n$ is $\overline{Z}^{1:n}$-measurable. 

\medskip

We are now in position to state our main result, the proof of which is postponed to Section~\ref{sec:proof}.  
\begin{theorem}\label{th:main}
Let $0< \delta <1$ and $(\delta_n)_{n\in\mathbb{N}^*} \subset (0,1)^{\mathbb{N}^*}$ be a sequence of numbers satisfying $\Pi_{n\in\mathbb{N}^*} \left( 1- \delta_n\right) \geq 1- \delta$. 
Let us assume that $\mathcal M$ satisfies assumption (B) and that $\nu$ satisfies assumption (A). Let $C,c>0$ be the constants defined in Corollary~\ref{cor:concentration}.

For all $n\in \mathbb{N}^*$, let 
\begin{equation}\label{eq:defKK}
K_\infty^{n}:= \max\left( K_\infty, \|\gb_1\|_{L^\infty}, \cdots, \|\gb_n\|_{L^\infty}\right) \quad \mbox{ and } \quad  K_{\mathcal L}^{n}:= \max\left( K_{\mathcal L}, \|\gb_1\|_{\mathcal L}, \cdots, \|\gb_n\|_{\mathcal L}\right).
\end{equation}
Let us assume that there exists $0< \gamma < 1$ such that for all $n\in \mathbb{N}^*$, $M_n \in\mathbb{N}^*$ is a $\overline{Z}^{1:(n-1)}$ measurable random variable which satisfies almost surely the following condition:
\begin{equation}\label{eq:cond1}
\forall n\geq 1, \quad M_n \geq -\ln\left( \frac{\delta_n}{C}\right) \frac{1}{c\phi\left( \kappa_{n-1}\right)},
\end{equation}
where $\kappa_{n-1}$ is a deterministic function of $\overline{Z}^{1:(n-1)}$, defined by 
\begin{equation}\label{eq:kappa0}
\kappa_0:= \frac{\left(1 - \gamma^2\right)  \sigh_0(\mathcal M)^2 }{8 K_\infty K_{\mathcal L} };
\end{equation}
 and 
\begin{equation}\label{eq:kappa1}
\forall n\geq 2, \quad \kappa_{n-1}:=  \frac{\min\left( \frac{1}{2(n-1)}, \quad \frac{(1-\gamma^2) \widehat{\sigma}_{n-1}(\mathcal M)^2}{n\left(9 K_2^2 + 4\right)}\right)}{6K_\infty^{n-1}  K_{\mathcal L}^{n-1}}.
\end{equation}
Then, for all $n\in \mathbb{N}^*$, it holds that
\begin{equation}\label{eq:firstrel}
\mathbb{P}\left[ \sigb_{n-1}(\mathcal M) \geq \gamma  \sigh_{n-1}(\mathcal M) \left| \overline{Z}^{1:(n-1)} \right. \right] \geq 1- \delta_n. 
\end{equation}
As a consequence, denoting by $\mathcal G_n$ the event $\sigb_{n-1}(\mathcal M) \geq \gamma  \sigh_{n-1}(\mathcal M)$ for all $n\in \mathbb{N}^*$, it holds that 
\begin{equation}\label{eq:lastres}
\mathbb{P}\left[ \bigcap_{n\in\mathbb{N}^*} \mathcal G_n\right] \geq 1- \delta.
\end{equation}
Thus, it then holds that the MC-greedy algorithm is a weak greedy algorithm with parameter $\gamma$ and norm $\|\cdot\|_{\Zb^{1:\infty}}$ with probability at least $1-\delta$. 
\end{theorem}

We state here a direct corollary of Theorem~\ref{th:main}, the proof of which is given below. 

\begin{corollary}\label{cor:main}
 Under the assumptions of Theorem~\ref{th:main}, with probability $1-\delta$, it holds that for all $n\in \mathbb{N}^*$, 
 \begin{equation}\label{eq:eq1n}
 \sigh_{n}(\mathcal M)  \leq \sqrt{2}\gamma^{-1} \mathop{\min}_{1\leq m < n} \left( d_m(\mathcal M) \right)^{\frac{n-m}{n}}.
 \end{equation}
 In particular, with probability $1-\delta$, it holds that 
 \begin{equation}\label{eq:eq2n}
 \forall n\in \mathbb{N}^*, \quad \sigh_{2n}(\mathcal M)  \leq \sqrt{2}\gamma^{-1} \sqrt{d_n(\mathcal M)}.
 \end{equation}
\end{corollary}

\begin{proof}
With probability $1-\delta$, the MC-greedy algorithm is a weak greedy algorithm with parameter $\gamma$ and norm $\|\cdot\|_{\Zb^{1:\infty}}$. Thus, since for all $n\in \mathbb{N}^*$, $\overline{\mu}_n$ is a 
$\Zb^{1:\infty}$ measurable random variable, if such an event is realized, using Theorem~\ref{cor:weakgreed}, it holds that for all $n\in \mathbb{N}^*$
 $$
  \sigh_{n}(\mathcal M)  \leq \sqrt{2}\gamma^{-1} \mathop{\min}_{1\leq m < n} \left( d_m^{\Zb^{1:\infty}}(\mathcal M) \right)^{\frac{n-m}{n}}, 
 $$
 where for all $n\in \mathbb{N}^*$, 
\begin{align*}
d_n^{\Zb^{1:\infty}}(\mathcal M)  & := \mathop{\inf}_{
 \begin{array}{c}
  V_n \subset L^2_{\nu,0}(\R^d) \; {\rm subspace},\\
  {\rm dim}\; V_n = n\\
 \end{array}
 }
 \mathop{\sup}_{\mu \in \mathcal P}
 \mathop{\inf}_{g_n\in V_n} \sqrt{{\rm Var}\left[ g_\mu(Z) - g_n(Z)\left|\Zb^{1:\infty}\right.\right]}\\ 
 & = \mathop{\inf}_{
\begin{array}{c}
  V_n \subset L^2_{\nu,0}(\R^d) \; {\rm subspace},\\
  {\rm dim}\; V_n = n\\
 \end{array}
 }
 \mathop{\sup}_{\mu \in \mathcal P}
 \mathop{\inf}_{g_n\in V_n} \sqrt{{\rm Var}\left[ g_\mu(Z) - g_n(Z)\right]}\\
 & = d_n(\mathcal M).\\
 \end{align*}
Hence, we obtain (\ref{eq:eq1n}), and (\ref{eq:eq2n}) as a consequence. 
\end{proof}
%
%
%

Some remarks are in order here. 

\begin{remark}
Note that, since the random variables $K_\infty^{n-1}$, $K_{\mathcal L}^{n-1}$ and $\widehat{\sigma}_{n-1}(\mathcal M)$ are measurable with respect to $\overline{Z}^{1:(n-1)}$, $\kappa_{n-1}$ 
is also measurable with respect to $\overline{Z}^{1:(n-1)}$.
\end{remark}

\begin{remark}\label{rem:main}
A natural question is then the following: can Theorem~\ref{th:main} be used (at least in principle) to design a \normalfont constructive strategy \normalfont\itshape  to choose a number of 
samples $M_n$, so that the MC-greedy algorithm can be guaranteed to be a weak greedy 
algorithm with parameter~$\gamma$? This can indeed be done in principle using the following remark: for all $n\in \mathbb{N}^*$, the quantity
$\sigh_{n-1}(\mathcal M)$ defined by (\ref{eq:sigh}) cannot be computed in practice since variances cannot be computed exactly for any parameter $\mu\in \mathcal P$. However, almost surely, it holds that
$\sigb_{n-1}(\mathcal M)$ defined by (\ref{eq:sigb}) satisfies $\sigb_{n-1}(\mathcal M)  \leq \sigh_{n-1}(\mathcal M)$. Let us recall that $\sigb_{n-1}(\mathcal M)$ depends on $\overline{Z}^{1:n}$, whereas 
$\sigh_{n-1}(\mathcal M)$ only depends on $\overline{Z}^{1:(n-1)}$. Since $\phi$ is an increasing function, this implies that, if the sequence $\left(M_n\right)_{n\in \mathbb{N}^*}$ satisfies condition (\ref{eq:cond1}) where 
$\sigh_0(\mathcal M)$ is replaced by $\sigb_0(\mathcal M)$ in (\ref{eq:kappa0}) and 
$\sigh_{n-1}(\mathcal M)$ is replaced by $\sigb_{n-1}(\mathcal M)$ in (\ref{eq:kappa1}), the assumptions of Theorem~\ref{th:main} are satisfied. Besides, it is reasonable to expect in this case 
that $\sigb_{n-1}(\mathcal M)$ should provide a reasonable approximation of $\sigh_{n-1}(\mathcal M)$ since, from Theorem~\ref{th:main}, $\sigb_{n-1}(\mathcal M)  \geq \gamma \sigh_{n-1}(\mathcal M)$ with high probability. 

\medskip

Unfortunately, we will see that such an approach is not viable in practice, because it leads to much too large values of $M_n$ for small values of $n$ for the MC-greedy algorithm to be interesting with a view 
to the variance reduction method 
explained in Section~\ref{sec:motiv}. That is why in Section~\ref{sec:resnum}, we will present numerical results with heuristic ways to choose values of $(M_n)_{n\in\mathbb{N}^*}$ which are not theoretically guaranteeed, but which nevertheless yield satisfactory numerical results in several test cases. 
\end{remark}

\subsection{Proof of Theorem~\ref{th:main}}\label{sec:proof}

The aim of this section is to prove Theorem~\ref{th:main}. For all $n\in \mathbb{N}^*$, we denote by $\mathcal G_n$ the event $\sigb_n(\mathcal M) \geq \gamma \sigh_n(\mathcal M)$.

Let us begin by proving some intermediate results which will be used later. We first need the following auxiliary lemma. 
\begin{lemma}\label{lem:cent}
 Let $n\in \mathbb{N}^*$. Then, almost surely,
$$
 \mathop{\sup}_{\begin{array}{c}
                 f \; \Zb^{1:\infty}\mbox{-measurable random function}\\
                 \mbox{such that } \|f\|_{\mathcal L} \leq 1 \mbox{ almost surely}\\
                \end{array}
} \left| \mathbb{E}\left[ f(Z) |\Zb^{1:\infty}\right] - \mathbb{E}_{\overline{Z}^n}(f)\right|=  \mathop{\sup}_{f\in \mathcal L; \|f\|_{\mathcal L} \leq 1} \left| \mathbb{E}\left[ f(Z)\right] - 
 \mathbb{E}_{\overline{Z}^n}(f)\right|.
$$
 \end{lemma}
 
\begin{proof}
On the one hand, it is obvious to check that
$$
\mathop{\sup}_{f\in \mathcal L; \|f\|_{\mathcal L} \leq 1} \left| \mathbb{E}\left[ f(Z)\right] - 
 \mathbb{E}_{\overline{Z}^n}(f)\right| \leq  \mathop{\sup}_{\begin{array}{c}
                 f \; \Zb^{1:\infty}\mbox{-measurable random function}\\
                 \mbox{such that } \|f\|_{\mathcal L} \leq 1 \mbox{ almost surely}\\
                \end{array}
} \left| \mathbb{E}\left[ f(Z) |\Zb^{1:\infty}\right] - \mathbb{E}_{\overline{Z}^n}(f)\right|.
$$

On the other hand, for any  $\Zb^{1:\infty}$-measurable random function $f$ such that $\|f\|_{\mathcal L} \leq 1$ almost surely, it holds that, almost surely, since $\Zb^{1:\infty}$ is independent of $Z$, 
$\mathbb{E}[f(Z) | \Zb^{1:\infty}] = \mathbb E_Z[ f(Z)]$ where the index $Z$ in $\mathbb{E}_Z$ indicates that the expectation is only taken with respect to $Z$, and thus
$$
\left| \mathbb{E}\left[ f(Z) |\Zb^{1:\infty}\right] - \mathbb{E}_{\overline{Z}^n}(f)\right| \leq \mathop{\sup}_{f\in \mathcal L; \|f\|_{\mathcal L} \leq 1} \left| \mathbb{E}\left[ f(Z)\right] - 
 \mathbb{E}_{\overline{Z}^n}(f)\right|.
$$
Hence the result. 
\end{proof}

We start by considering the case of the initialization of the MC-greedy algorithm.
\begin{lemma}\label{lem:first}
Let $0<\gamma<1$. Then, it holds that almost surely,
\begin{equation}\label{eq:est}
\mathbb{P}\left[ {\rm Var}\left[g_{\mub_1}(Z) \left|\Zb^{1:\infty}\right.\right] \geq \gamma^2 \mathop{\max}_{\mu \in \mathcal P} {\rm Var}\left[g_{\mu}(Z) \left| \Zb^{1:\infty}\right.\right]\right] \geq 1-\delta_1. 
\end{equation}
As a consequence, $\mathbb{P}\left[ \mathcal G_1 \right] \geq 1-\delta_1$ and (\ref{eq:firstrel}) holds for $n=1$. 
\end{lemma}

\begin{proof}
 Let $\muh_1 \in \mathcal P$ 
 such that $$
 \sigh_0(\mathcal M)^2= \mathop{\max}_{\mu \in \mathcal P} {\rm Var}\left[g_{\mu}(Z) \left| \Zb^{1:\infty}\right.\right] = {\rm Var}\left[g_{\muh_1}(Z)| \Zb^{1:\infty}\right].
 $$
 Inequality (\ref{eq:est}) holds provided that
 $$
\mathbb{P}\left[\left( {\rm Var}\left[g_{\muh_1}(Z)|\Zb^{1:\infty}\right] - {\rm Var}\left[g_{\mub_1}(Z) \left|\Zb^{1:\infty}\right.\right] \right)> \epsilon  \sigh_0(\mathcal M)^2  \right] \leq \delta_1,
 $$
 with $\epsilon:= \left(1 - \gamma^2\right)$.
Almost surely, since $\mub_1\in \mathop{{\rm argmax}}_{\mu \in \mathcal P} {\rm Var}_{\Zb^1}(g_\mu)$, it holds that
 \begin{align*}
 & {\rm Var}\left[g_{\muh_1}(Z)|\Zb^{1:\infty}\right] - {\rm Var}\left[g_{\mub_1}(Z) \left|\Zb^{1:\infty}\right.\right]\\
 &=   {\rm Var}\left[g_{\muh_1}(Z)|\Zb^{1:\infty}\right] - {\rm Var}_{\Zb^1}\left(g_{\muh_1}\right) +{\rm Var}_{\Zb^1}\left(g_{\muh_1}\right) - {\rm Var}_{\Zb^1}\left(g_{\mub_1}\right) + {\rm Var}_{\Zb^1}\left(g_{\mub_1}\right) - {\rm Var}\left[g_{\mub_1}(Z) \left|\Zb^{1:\infty}\right.\right] \\
  & \leq {\rm Var}\left[g_{\muh_1}(Z)|\Zb^{1:\infty}\right] - {\rm Var}_{\Zb^1}\left(g_{\muh_1}\right) + {\rm Var}_{\Zb^1}\left(g_{\mub_1}\right) - {\rm Var}\left[g_{\mub_1}(Z) \left|\Zb^{1:\infty}\right.\right] \\
  & = \mathbb{E}\left[|g_{\muh_1}|^2(Z) |\Zb^{1:\infty}\right] - \mathbb{E}_{\Zb^1}\left( |g_{\muh_1}|^2 \right) + \mathbb{E}_{\Zb^1}\left( g_{\muh_1}\right)^2 - \mathbb{E}\left[g_{\muh_1}(Z)\left| \Zb^{1:\infty}\right.\right]^2 \\
  &\quad -\mathbb{E}\left[|g_{\mub_1}|^2(Z)|\Zb^{1:\infty}\right] + \mathbb{E}_{\Zb^1}\left( |g_{\mub_1}|^2 \right) - \mathbb{E}_{\Zb^1}\left( g_{\mub_1}\right)^2 + \mathbb{E}\left[g_{\mub_1}(Z)|\Zb^{1:\infty}\right]^2 \\
  & \leq  \left|\mathbb{E}\left[|g_{\muh_1}|^2(Z)|\Zb^{1:\infty}\right] - \mathbb{E}_{\Zb^1}\left( |g_{\muh_1}|^2 \right)\right| + 2K_\infty \left|\mathbb{E}_{\Zb^1}\left( g_{\muh_1}\right) - \mathbb{E}\left[g_{\muh_1}(Z)\right]\right| \\
 & \quad + \left|\mathbb{E}\left[|g_{\mub_1}|^2(Z)|\Zb^{1:\infty}\right] - \mathbb{E}_{\Zb^1}\left( |g_{\mub_1}|^2 \right)\right| + 2K_\infty\left| \mathbb{E}_{\Zb^1}\left( g_{\mub_1}\right) - \mathbb{E}\left[g_{\mub_1}(Z)|\Zb^{1:\infty}\right]\right|,\\
 & \leq  2 K_\infty K_{\mathcal L}\\
 & \times \left( \left|\mathbb{E}\left[\frac{|g_{\muh_1}|^2}{2K_\infty K_{\mathcal L}}(Z)| \Zb^{1:\infty}\right] - \mathbb{E}_{\Zb^1}\left( \frac{|g_{\muh_1}|^2}{2K_\infty K_{\mathcal L}} \right)\right| 
 + \left|\mathbb{E}\left[\frac{|g_{\mub_1}|^2}{2K_\infty K_{\mathcal L}}(Z)| \Zb^{1:\infty}\right] - \mathbb{E}_{\Zb^1}\left( \frac{|g_{\mub_1}|^2}{2K_\infty K_{\mathcal L}} \right)\right|\right.\\
 &\quad \left. 
 + \left| \mathbb{E}_{\Zb^1}\left( \frac{g_{\mub_1}}{K_{\mathcal L}}\right) - \mathbb{E}\left[\frac{g_{\mub_1}}{K_\mathcal L}(Z)|\Zb^{1:\infty}\right]\right| +  \left| \mathbb{E}_{\Zb^1}\left( \frac{g_{\muh_1}}{K_{\mathcal L}}\right) - \mathbb{E}\left[\frac{g_{\muh_1}}{K_\mathcal L}(Z)|\Zb^{1:\infty}\right]\right|\right).
 \end{align*}
It holds that for all $\mu \in \mathcal P$, $\||g_{\mu}|^2 \|_{\mathcal L}\leq 2 K_\infty K_{\mathcal L}$. Indeed, for all $x,y\in\R^d$, we have
$$
| |g_{\mu}|^2(x) - |g_{\mu}|^2(y)| = |(g_{\mu}(x) + g_{\mu}(y))(g_{\mu}(x) - g_{\mu}(y))| \leq 2 K_\infty K_{\mathcal L}|x-y|.
$$
Thus, almost surely, it holds that
$$
{\rm Var}\left[g_{\muh_1}(Z)|\Zb^{1:\infty}\right] - {\rm Var}\left[g_{\mub_1}(Z) \left|\Zb^{1:\infty}\right.\right] \leq 8 K_\infty K_{\mathcal L} \sup_{f\in \mathcal L; \|f\|_{\mathcal L} \leq 1} \left|\mathbb{E}\left[f(Z)| \Zb^{1:\infty}\right] - \mathbb{E}_{\Zb^1}\left( f \right)\right|.
$$
Then, using Lemma~\ref{lem:cent}, we obtain that, almost surely,
$$
{\rm Var}\left[g_{\muh_1}(Z)|\Zb^{1:\infty}\right] - {\rm Var}\left[g_{\mub_1}(Z) \left|\Zb^{1:\infty}\right.\right] \leq 8 K_\infty K_{\mathcal L} \sup_{f\in \mathcal L; \|f\|_{\mathcal L} \leq 1} \left|\mathbb{E}\left[f(Z)\right] - \mathbb{E}_{\Zb^1}\left( f \right)\right|.
$$
Thus, using Theorem~\ref{cor:concentration}, the assumption on $M_1$ and the definition of $\kappa_0$, we obtain that
$$
\mathbb{P}\left[ \sup_{f \in \mathcal L; \|f\|_{\mathcal L} \leq 1} \left|\mathbb{E}\left[f(Z)\right] - \mathbb{E}_{\Zb^1}\left( f \right)\right| \geq \kappa_0 \right] \leq C e^{-c \phi(\kappa_0)} \leq \delta_1.
$$
Hence the desired result. 
\end{proof}

We now turn to the case of the $n^{th}$ iteration of the algorithm, with $n\geq 2$, that we analyze in the next two lemmas.
\begin{lemma}\label{lem:cov}
Let $n\geq 2$. Let us denote by 
\begin{equation}\label{eq:defMn}
\mathcal{M}^{n-1} := \mathcal M \cup \{ \gb_1, \cdots, \gb_{n-1}\}.
\end{equation}
Then, for all $\epsilon>0$, it holds that, almost surely, 
$$
\mathbb{P}\left[\left. \sup_{g,h \in \mathcal{M}^{n-1}} \left| {\rm Cov}\left[g(Z), h(Z)\left|\Zb^{1:\infty}\right.\right] - {\rm Cov}_{\overline{Z}^n}\left(g, h\right)\right| \geq \epsilon \right| \Zb^{1:(n-1)}\right] 
\leq Ce^{-c M_n \phi\left(\frac{\epsilon}{  6K_\infty^{n-1}   K_{\mathcal L}^{n-1} }\right)}, 
$$
where $K_{\mathcal L}^{n-1}$ and $K_{\infty}^{n-1}$ are defined by (\ref{eq:defKK}). 
\end{lemma}

\begin{proof}
For all $g,h\in \mathcal{M}^{n-1}$, it holds that, almost surely,
\begin{align*}
 & \left| {\rm Cov}\left[g(Z), h(Z)\left|\Zb^{1:\infty}\right.\right] - {\rm Cov}_{\overline{Z}^n}\left( g, h\right)\right| \\
 & \leq \left| \mathbb{E}\left[g(Z)h(Z)\left|\Zb^{1:\infty}\right.\right] - \mathbb{E}_{\overline{Z}^n}\left(gh\right)\right|\\
 & \quad + K_\infty^{n-1} \left( \left| \mathbb{E}\left[g(Z)\left|\Zb^{1:\infty}\right.\right] - \mathbb{E}_{\overline{Z}^n}\left( g\right)\right| + \left| \mathbb{E}\left[h(Z)\left|\Zb^{1:\infty}\right.\right] 
 - \mathbb{E}_{\overline{Z}^n}\left( h\right)\right| \right)\\
 & \leq  2K_\infty^{n-1}   K_{\mathcal L}^{n-1} \left( \left| \mathbb{E}\left[\frac{gh}{2K_\infty^{n-1}   K_{\mathcal L}^{n-1}}(Z)\left|\Zb^{1:\infty}\right.\right] - 
 \mathbb{E}_{\overline{Z}^n}\left( \frac{gh}{2K_\infty^{n-1}   K_{\mathcal L}^{n-1}}\right)\right|\right.\\
 & \quad + \left. \left| \mathbb{E}\left[\frac{g}{2  K_{\mathcal L}^{n-1}}(Z)\left|\Zb^{1:\infty}\right.\right] - \mathbb{E}_{\overline{Z}^n}\left( \frac{g}{2  K_{\mathcal L}^{n-1}}\right)\right| + 
 \left| \mathbb{E}\left[\frac{h}{2  K_{\mathcal L}^{n-1}}(Z)\left|\Zb^{1:\infty}\right.\right] - \mathbb{E}_{\overline{Z}^n}\left( \frac{h}{2  K_{\mathcal L}^{n-1}}\right)\right| \right).\\ 
\end{align*}
For all $g,h\in \mathcal{M}^{n-1}$, it holds that 
$$
\left\| \frac{gh}{2K_\infty^{n-1}   K_{\mathcal L}^{n-1}} \right\|_{\mathcal L}\leq 1 \quad \mbox{ and } \left\| \frac{g}{2  K_{\mathcal L}^{n-1}} \right\|_{\mathcal L}\leq 1. 
$$
This implies that, almost surely, 
$$
\sup_{g,h \in \mathcal{M}^{n-1}}\left| {\rm Cov}\left[g(Z), h(Z)\left|\Zb^{1:\infty}\right.\right] - {\rm Cov}_{\overline{Z}^n}\left( g, h\right)\right| 
\leq 6K_\infty^{n-1}   K_{\mathcal L}^{n-1}  \sup_{f\in \mathcal L, \; \|f\|_{\mathcal L}\leq 1}  \left| \mathbb{E}\left[f(Z)\left|\Zb^{1:\infty}\right.\right] - \mathbb{E}_{\overline{Z}^n}\left( f\right)\right|.
$$
Using Lemma~\ref{lem:cent}, this yields that, almost surely, 
$$\sup_{g,h \in \mathcal{M}^{n-1}}\left| {\rm Cov}\left[g(Z), h(Z)\left|\Zb^{1:\infty}\right.\right] - {\rm Cov}_{\overline{Z}^n}\left( g, h\right)\right| 
\leq 6K_\infty^{n-1}   K_{\mathcal L}^{n-1}  \sup_{f\in \mathcal L, \; \|f\|_{\mathcal L}\leq 1}  \left| \mathbb{E}\left[f(Z)\right] - \mathbb{E}_{\overline{Z}^n}\left( f\right)\right|.
$$
We finally obtain, using Corollary~\ref{cor:concentration}, that
\begin{align*}
& \mathbb{P}\left[ \left. \sup_{g,h \in \mathcal{M}^{n-1}} \left| {\rm Cov}\left[g(Z),h(Z)\left|\Zb^{1:\infty}\right.\right] - {\rm Cov}_{\overline{Z}^n}\left( g, h\right)\right| > \epsilon \right| \Zb^{1:(n-1)}\right]  \\
& \leq \mathbb{P}\left[ \left. \sup_{f\in \mathcal L, \; \|f\|_{\mathcal L}\leq 1}  \left| \mathbb{E}\left[f(Z)\right] - \mathbb{E}_{\overline{Z}^n}\left( f\right)\right| > \frac{\epsilon}{  6K_\infty^{n-1}   K_{\mathcal L}^{n-1} }\right| \Zb^{1:(n-1)}\right]\\
& \leq \mathbb{P}\left[ \left. \sup_{f\in \mathcal L, \; \|f\|_{\mathcal L}\leq 1}  \left| \mathbb{E}\left[f(Z)\right] - \mathbb{E}_{\overline{Z}^n}\left( f\right)\right| >  \frac{\epsilon}{  6K_\infty^{n-1}   K_{\mathcal L}^{n-1} }\right| \Zb^{1:(n-1)}\right]\\
& \leq Ce^{-c M_n \phi\left(\frac{\epsilon}{  6K_\infty^{n-1}   K_{\mathcal L}^{n-1} }\right)}.\\
\end{align*}
Hence the result.
\end{proof}

\begin{lemma}\label{lem:nit}
Let $0<\gamma<1$ and $n\geq 2$. Then, it holds that almost surely
$$
\mathbb{P}\left[\mathcal{G}_n |\Zb^{1:(n-1)} \right]\geq 1-\delta_n.
$$
\end{lemma}

\begin{proof}
Since $(\gb_1, \cdots, \gb_n)$ forms a basis of $\Vb_{n-1}$, for all $\mu\in\mathcal P$, there exists one unique minimizer to 
$$
\mathop{\min}_{\left( \lambda_i\right)_{1\leq i \leq n-1}\in \R^{n-1} }  {\rm Var}\left[ \left. g_{\mu}(Z) - \sum_{i=1}^{n-1} \lambda_i \gb_i(Z)\right| \Zb^{1:\infty}\right].
$$
Let $\lambb^n:=\left( \lambb^n_i\right)_{1\leq i \leq n-1} \in \R^{n-1}$ be the unique minimizer of 
\begin{equation}\label{eq:lambbn}
\lambb^n: = \mathop{\rm argmin}_{\lambda:=(\lambda_i)_{1\leq i \leq n-1}}  {\rm Var}\left[ \left. g_{\mub_n}(Z) - \sum_{i=1}^{n-1} \lambda_i \gb_i(Z)\right| \Zb^{1:\infty}\right].
\end{equation}
As a consequence, it holds that 
$$
\sigb_{n-1}(\mathcal M) = {\rm Var}\left[ \left. g_{\overline{\mu}_n}(Z) - \sum_{i=1}^{n-1} \lambb^n_i \gb_i(Z)\right| \Zb^{1:\infty}\right]
$$
where $\sigb_{n-1}(\mathcal M)$ is defined by (\ref{eq:sigb}).

Let $\muh_n\in \mathcal P$ such that 
$$
\muh_n \in \mathop{\rm argmax}_{\mu\in \mathcal P} \mathop{\min}_{\left( \lambda_i\right)_{1\leq i \leq n-1}\in \R^{n-1} }  {\rm Var}\left[ \left. g_{\mu}(Z) - \sum_{i=1}^{n-1} \lambda_i \gb_i(Z)\right| \Zb^{1:\infty}\right],
$$
so that 
$$
\sigh_{n-1}(\mathcal M) = \mathop{\min}_{\left( \lambda_i\right)_{1\leq i \leq n-1}\in \R^{n-1} }  {\rm Var}\left[ \left. g_{\muh_n}(Z) - \sum_{i=1}^{n-1} \lambda_i \gb_i(Z)\right| \Zb^{1:\infty}\right],
$$
where $\sigh_{n-1}(\mathcal M)$ is defined in (\ref{eq:sigh}).

\medskip

Let $\lambh^n:=\left( \lambh^n_i\right)_{1\leq i \leq n-1} \in \R^{n-1}$ the unique minimizer of
\begin{equation}\label{eq:lambhn}
\lambh^n:=\mathop{\rm argmin}_{\lambda:=(\lambda_i)_{1\leq i \leq n-1}} {\rm Var}\left[ \left. g_{\muh_n}(Z)- \sum_{i=1}^{n-1} \lambda_i \gb_i(Z)\right| \Zb^{1:\infty}\right],
\end{equation}
so that
$$
\sigh_{n-1}(\mathcal M) = {\rm Var}\left[ \left. g_{\muh_n}(Z) - \sum_{i=1}^{n-1} \lambh^n_i \gb_i(Z)\right| \Zb^{1:\infty}\right].
$$

\medskip

The event $\mathcal G_n$ holds if and only if 
\begin{equation}\label{eq:ineq00}
 {\rm Var}\left[ \left. g_{\overline{\mu}_n}(Z) - \sum_{i=1}^{n-1} \lambb^n_i \gb_i(Z)\right| \Zb^{1:\infty}\right] \geq \gamma^2   {\rm Var}\left[ \left. g_{\muh_n}(Z) - \sum_{i=1}^{n-1} \lambh^n_i \gb_i(Z)\right| \Zb^{1:\infty}\right]. 
\end{equation}

Let us begin by pointing out that, since 
$$
{\rm Var}\left[ \left. g_{\muh_n}(Z) - \sum_{i=1}^{n-1} \lambh^n_i \gb_i(Z)\right| \Zb^{1:\infty}\right] \geq  
{\rm Var}\left[ \left. g_{\overline{\mu}_n}(Z) - \sum_{i=1}^{n-1} \lambb^n_i \gb_i(Z)\right| \Zb^{1:\infty}\right],
$$
if the inequality
\begin{align}\nonumber
 & {\rm Var}\left[ \left. g_{\muh_n}(Z) - \sum_{i=1}^{n-1} \lambh^n_i \gb_i(Z)\right| \Zb^{1:\infty}\right]  - 
 {\rm Var}\left[ \left. g_{\overline{\mu}_n}(Z) - \sum_{i=1}^{n-1} \lambb^n_i \gb_i(Z)\right| \Zb^{1:\infty}\right]  
 \\ \label{eq:ineq0}
 & \leq (1-\gamma^2) {\rm Var}\left[ \left. g_{\muh_n}(Z) - \sum_{i=1}^{n-1} \lambh^n_i \gb_i(Z)\right| \Zb^{1:\infty}\right]\\ \nonumber
\end{align}
holds, then (\ref{eq:ineq00}) is necessarily statisfied. The rest of the proof consists in estimating the probability that (\ref{eq:ineq0}) is realized. 

To this aim, as a first step, 
we are going to prove an upper bound on 
$$
{\rm Var}\left[ \left. g_{\muh_n}(Z) - \sum_{i=1}^{n-1} \lambh^n_i \gb_i(Z)\right| \Zb^{1:\infty}\right]  - 
 {\rm Var}\left[ \left. g_{\overline{\mu}_n}(Z) - \sum_{i=1}^{n-1} \lambb^n_i \gb_i(Z)\right| \Zb^{1:\infty}\right]
 $$
as a function of  
\begin{equation}\label{eq:eta}
\eta:= \sup_{g,h\in \mathcal{M}_{n-1}} \left| {\rm Cov}\left[ \left. g(Z), h(Z) \right| \Zb^{1:\infty}\right] - {\rm Cov}_{\overline{Z}^n}\left(  g, h \right)  \right|,
\end{equation}
which is the quantity estimated in Lemma~\ref{lem:cov}. 
More precisely, let us now prove that
  \begin{align}\nonumber
   &{\rm Var}\left[ \left. g_{\muh_n}(Z)- \sum_{i=1}^{n-1} \lambh^n_i \gb_i(Z)\right| \Zb^{1:\infty}\right] -  {\rm Var}\left[ \left. g_{\mub_n}(Z)- \sum_{i=1}^{n-1} \lambb^n_i \gb_i(Z)\right| \Zb^{1:\infty}\right] \\ \label{ineqcentr}
  & \leq n\left( 2 + \left(\frac{K_2 + \sqrt{n-1}\eta}{1-(n-1)\eta}\right)^2 + K_2^2\right)\eta.\\\nonumber
\end{align}

It holds that for all $1\leq i \leq n-1$, from (\ref{eq:lambbn}) and (\ref{eq:lambhn}), 
$$
\lambh^n_i={\rm Cov}\left[ g_{\muh_n}(Z), \gb_i(Z) |\Zb^{1:\infty}\right]  \quad \mbox{ and } \lambb^n_i={\rm Cov}\left[ g_{\mub_n}(Z), \gb_i(Z) |\Zb^{1:\infty}\right],
$$
and it then holds that, almost surely,
\begin{equation}\label{eq:K2}
\max\left( \|\lambh^n\|_{\ell^2}, \|\lambb^n\|_{\ell^2} \right) \leq \max\left( \|g_{\muh_n}\|, \|g_{\mub_n}\| \right) \leq K_2,
\end{equation}
where $\|\cdot\|_{\ell^2}$ denotes the Euclidean norm of $\R^{n-1}$. Let now $\lambh^{n,n}:=\left( \lambh^{n,n}_i\right)_{1\leq i \leq n-1} \in \R^{n-1}$ be a minimizer of
$$
\lambh^{n,n}:=\mathop{\rm argmin}_{\lambda:=(\lambda_i)_{1\leq i \leq n-1} \in \R^{n-1}} {\rm Var}_{\overline{Z}^n}\left(  g_{\muh_n}- \sum_{i=1}^{n-1} \lambda_i \gb_i\right),
$$
and $\lambb^{n,n}:=\left( \lambb^{n,n}_i\right)_{1\leq i \leq n-1} \in \R^{n-1}$ be a minimizer of 
$$
\lambb^{n,n}: = \mathop{\rm argmin}_{\lambda:=(\lambda_i)_{1\leq i \leq n-1}\in \R^{n-1}}  {\rm Var}_{\overline{Z}^n}\left(  g_{\mub_n} - \sum_{i=1}^{n-1} \lambda_i \gb_i\right)
$$
It then holds that for all $1\leq i \leq n-1$, $\lambh^{n,n}$ and $\lambb^{n,n}$ are solution to the linear systems
$$
A^n \lambh^{n,n} = \bh^n \quad \mbox{ and } A^n \lambb^{n,n} = \bb^n,
$$
where 
$A^n:=\left( A^n_{ij}\right)_{1\leq i,j \leq n-1}\in \R^{(n-1)\times (n-1)}$, $\bh^n:=\left(\bh^n_i\right)_{1\leq i\leq n-1},\bb^n:=\left(\bb^n_i\right)_{1\leq i\leq n-1}  \in \R^{n-1}$ are defined as follows: for all $1\leq i,j \leq n-1$, 
$$
A^n_{ij} = {\rm Cov}_{\overline{Z}^n}\left(\gb_i, \gb_j\right), \; \bh^n_i =  {\rm Cov}_{\overline{Z}^n}\left(g_{\muh_n}, \gb_i\right) \; \mbox{ and }\bb^n_i =  {\rm Cov}_{\overline{Z}^n}\left(g_{\mub_n}, \gb_i\right).
$$

Then, it holds that, almost surely, 
$$
\max_{1\leq i \leq n-1}\left(\left|\bh_i^n - \lambh_i^n\right|, \left|\bb^n_i- \lambb^n_i\right|\right) \leq \eta,
$$
which implies that
$$
\max\left( \|\bh_n\|_{\ell^2},\|\bb_n\|_{\ell^2}\right) \leq K_2 + \sqrt{n-1} \eta.
$$
Moreover, we have
$$
\max_{1\leq i,j \leq n-1} \left|A^n_{ij} - \delta_{ij} \right| \leq \eta,
$$
which yields that for all $\xi\in \R^{n-1}$,
$$
(1-(n-1)\eta) \|\xi\|_{\ell^2}^2 \leq \xi^TA^n\xi \leq (1+(n-1)\eta)\|\xi\|_{\ell^2}^2.
$$
Assume for now that $\eta(n-1) <1$, this implies that, for all $\xi\in \R^{n-1}$,
\begin{equation}\label{eq:borne}
\|(A^n)^{-1} \xi\|_{\ell^2} \leq \frac{1}{1-(n-1)\eta}\|\xi\|_{\ell^2}.
\end{equation}

Using (\ref{eq:borne}), we obtain that
\begin{equation}\label{eq:K22}
\max\left( \|\lambb^{n,n} \|_{\ell^2}, \|\lambh^{n,n}\|_{\ell^2} \right) \leq \frac{K_2 + \sqrt{n-1}\eta}{1-(n-1)\eta}.
\end{equation}

We then have,
\begin{align*}
 &{\rm Var}\left[ \left. g_{\muh_n}(Z)- \sum_{i=1}^{n-1} \lambh^n_i \gb_i(Z)\right| \Zb^{1:\infty}\right] -  {\rm Var}\left[ \left. g_{\mub_n}(Z)- \sum_{i=1}^{n-1} \lambb^n_i \gb_i(Z)\right| \Zb^{1:\infty}\right] \\
  & = {\rm Var}\left[ \left. g_{\muh_n}(Z)- \sum_{i=1}^{n-1} \lambh^n_i \gb_i(Z)\right| \Zb^{1:\infty}\right] - {\rm Var}\left[ \left. g_{\muh_n}(Z)- \sum_{i=1}^{n-1} \lambh^{n,n}_i \gb_i(Z)\right| \Zb^{1:\infty}\right] \\
 & \quad + {\rm Var}\left[ \left. g_{\muh_n}(Z)- \sum_{i=1}^{n-1} \lambh^{n,n}_i \gb_i(Z)\right| \Zb^{1:\infty}\right] - {\rm Var}_{\overline{Z}^n}\left(  g_{\muh_n}- \sum_{i=1}^{n-1} \lambh^{n,n}_i \gb_i\right) \\
 & \quad + {\rm Var}_{\overline{Z}^n}\left( g_{\muh_n}- \sum_{i=1}^{n-1} \lambh^{n,n}_i \gb_i\right) - {\rm Var}_{\overline{Z}^n}\left(  g_{\mub_n}- \sum_{i=1}^{n-1} \lambb^{n,n}_i \gb_i\right) \\
 & \quad + {\rm Var}_{\overline{Z}^n}\left( g_{\mub_n}- \sum_{i=1}^{n-1} \lambb^{n,n}_i \gb_i\right) - {\rm Var}_{\overline{Z}^n}\left( g_{\mub_n}- \sum_{i=1}^{n-1} \lambb^{n}_i \gb_i\right) \\
  & \quad + {\rm Var}_{\overline{Z}^n}\left(  g_{\mub_n}- \sum_{i=1}^{n-1} \lambb^{n}_i \gb_i\right) - {\rm Var}\left[ \left. g_{\mub_n}(Z)- \sum_{i=1}^{n-1} \lambb^n_i \gb_i(Z)\right| \Zb^{1:\infty}\right].\\
  \end{align*}
  Using the fact that 
  \begin{align*}
  &{\rm Var}\left[ \left. g_{\muh_n}(Z)- \sum_{i=1}^{n-1} \lambh^n_i \gb_i(Z)\right| \Zb^{1:\infty}\right] - {\rm Var}\left[ \left. g_{\muh_n}(Z)- \sum_{i=1}^{n-1} \lambh^{n,n}_i \gb_i(Z)\right| \Zb^{1:\infty}\right] \leq 0,\\  
  & {\rm Var}_{\overline{Z}^n}\left( g_{\muh_n}- \sum_{i=1}^{n-1} \lambh^{n,n}_i \gb_i\right) - {\rm Var}_{\overline{Z}^n}\left(  g_{\mub_n}- \sum_{i=1}^{n-1} \lambb^{n,n}_i \gb_i\right) \leq 0,\\
  & {\rm Var}_{\overline{Z}^n}\left( g_{\mub_n}- \sum_{i=1}^{n-1} \lambb^{n,n}_i \gb_i\right) - {\rm Var}_{\overline{Z}^n}\left( g_{\mub_n}- \sum_{i=1}^{n-1} \lambb^{n}_i \gb_i\right) \leq 0,\\
  \end{align*}
from the definition of $\lambh^n$, $\lambh^{n,n}$, $\lambb^{n,n}$, $\mub_n$, we obtain that  
  \begin{align*}
   &{\rm Var}\left[ \left. g_{\muh_n}(Z)- \sum_{i=1}^{n-1} \lambh^n_i \gb_i(Z)\right| \Zb^{1:\infty}\right] -  {\rm Var}\left[ \left. g_{\mub_n}(Z)- \sum_{i=1}^{n-1} \lambb^n_i \gb_i(Z)\right| \Zb^{1:\infty}\right] \\
   & \leq {\rm Var}\left[ \left. g_{\muh_n}(Z)- \sum_{i=1}^{n-1} \lambh^{n,n}_i \gb_i(Z)\right| \Zb^{1:\infty}\right] - {\rm Var}_{\overline{Z}^n}\left(  g_{\muh_n}- \sum_{i=1}^{n-1} \lambh^{n,n}_i \gb_i\right) \\
   & \quad + {\rm Var}_{\overline{Z}^n}\left(  g_{\mub_n}- \sum_{i=1}^{n-1} \lambb^{n}_i \gb_i\right) - {\rm Var}\left[ \left. g_{\mub_n}(Z)- \sum_{i=1}^{n-1} \lambb^n_i \gb_i(Z)\right| \Zb^{1:\infty}\right], \\
   & =  {\rm Var}\left[ \left. g_{\muh_n}(Z)\right| \Zb^{1:\infty}\right] - {\rm Var}_{\overline{Z}^n}\left(  g_{\muh_n}\right) - 2\sum_{i=1}^{n-1} \lambh^{n,n}_i \left( {\rm Cov}\left[ \left. g_{\muh_n}(Z), \gb_i(Z) \right| \Zb^{1:\infty}\right] 
   -{\rm Cov}_{\overline{Z}^n}\left(  g_{\muh_n}, \gb_i \right)  \right) \\
 & \quad + \sum_{i,j=1}^{n-1} \lambh^{n,n}_i\lambh^{n,n}_j \left( {\rm Cov}\left[ \left. \gb_i(Z), \gb_j(Z) \right| \Zb^{1:\infty}\right]  -{\rm Cov}_{\overline{Z}^n}\left(\gb_i, \gb_j \right)  \right) \\
 & \quad + {\rm Var}_{\overline{Z}^n}\left(  g_{\mub_n}\right) - {\rm Var}\left[ \left. g_{\mub_n}(Z)\right| \Zb^{1:\infty}\right] - 2 \sum_{i=1}^{n-1}\lambb^{n}_i\left( {\rm Cov}_{\overline{Z}^n}\left(  g_{\mub_n}, \gb_i\right) -  {\rm Cov}\left[ \left. \gb_i(Z), \gb_j(Z)\right| \Zb^{1:\infty}\right] \right)\\
& \quad + \sum_{i,j=1}^{n-1}\lambb^{n}_i\lambb^{n}_j\left( {\rm Cov}_{\overline{Z}^n}\left( \gb_i, \gb_j\right) - {\rm Cov}\left[ \left. g_{\mub_n}(Z), \gb_i(Z)\right| \Zb^{1:\infty}\right] \right).\\
\end{align*}
Now, using the definition of $\mathcal{M}^{n-1}$ given in (\ref{eq:defMn}), we obtain that
\begin{align*}
   &{\rm Var}\left[ \left. g_{\muh_n}(Z)- \sum_{i=1}^{n-1} \lambh^n_i \gb_i(Z)\right| \Zb^{1:\infty}\right] -  {\rm Var}\left[ \left. g_{\mub_n}(Z)- \sum_{i=1}^{n-1} \lambb^n_i \gb_i(Z)\right| \Zb^{1:\infty}\right] \\
 & \leq \left(1 + 2\sum_{i=1}^{n-1}|\lambh^{n,n}_i| + \sum_{i,j=1}^{n-1} |\lambh^{n,n}_i||\lambh^{n,n}_j|  + 1 + 2\sum_{i=1}^{n-1}|\lambb^{n}_i| + \sum_{i,j=1}^{n-1} |\lambb^{n}_i||\lambb^{n}_j|\right) \\
 & \quad \times \sup_{g,h\in \mathcal{M}_{n-1}} \left| {\rm Cov}\left[ \left. g(Z), h(Z) \right| \Zb^{1:\infty}\right] - {\rm Cov}_{\overline{Z}^n}\left(  g, h \right)  \right|.\\
 \end{align*}
 Since $\displaystyle \mathop{\sup}_{g,h\in \mathcal{M}_{n-1}} \left| {\rm Cov}\left[ \left. g(Z), h(Z) \right| \Zb^{1:\infty}\right] - {\rm Cov}_{\overline{Z}^n}\left(  g, h \right)  \right|=\eta$, we then have, almost surely, 
 \begin{align*}
   &{\rm Var}\left[ \left. g_{\muh_n}(Z)- \sum_{i=1}^{n-1} \lambh^n_i \gb_i(Z)\right| \Zb^{1:\infty}\right] -  {\rm Var}\left[ \left. g_{\mub_n}(Z)- \sum_{i=1}^{n-1} \lambb^n_i \gb_i(Z)\right| \Zb^{1:\infty}\right] \\
 & \leq \left[ \left( 1 + \sum_{i=1}^{n-1}|\lambh^{n,n}_i|\right)^2  +  \left( 1 + \sum_{i=1}^{n-1}|\lambb^{n}_i|\right)^2\right] \sup_{g,h\in \mathcal{M}_{n-1}} \left| {\rm Cov}\left[ \left. g(Z), h(Z) \right| \Zb^{1:\infty}\right] - {\rm Cov}_{\overline{Z}^n}\left(  g, h \right)  \right|\\
 & \leq n\left( 2 + \sum_{i=1}^{n-1}|\lambh^{n,n}_i|^2 + \sum_{i=1}^{n-1}|\lambb^n_i|^2\right) \eta\\
 & \leq n\left( 2 +  \|\lambh^{n,n}\|_{\ell^2}^2 + \|\lambb^n\|_{\ell^2}^2\right)  \eta.\\
 \end{align*}
 Finally, using (\ref{eq:K2}) and (\ref{eq:K22}), we obtain (\ref{ineqcentr}), i.e. 
 $$
 \sigh_{n-1}(\mathcal M) - \sigb_{n-1}(\mathcal M) \leq n\left( 2 + \left(\frac{K_2 + \sqrt{n-1}\eta}{1-(n-1)\eta}\right)^2 + K_2^2\right)\eta.
 $$

\medskip

Let us now evaluate the probability, conditioned to the knowledge of $\Zb^{1:\infty}$, that 
\begin{equation}\label{eq:ineqaux}
n\left( 2 + \left(\frac{K_2 + \sqrt{n-1}\eta}{1-(n-1)\eta}\right)^2 + K_2^2\right)\eta \leq (1-\gamma^2) \widehat{\sigma}_{n-1}(\mathcal M).
\end{equation}
If $\eta$ is chosen to be smaller that $\frac{1}{2(n-1)}$, then it holds that 
$$
 2 + \left(\frac{K_2 + \sqrt{n-1}\eta}{1-(n-1)\eta}\right)^2 + K_2^2  \leq 2 + \left(2 K_2 + 1\right)^2 + K_2^2 \leq 9 K_2^2 + 4.
$$
A sufficient condition for (\ref{eq:ineqaux}) to hold is then to ensure that $\eta \leq \epsilon $ with
$$
\epsilon := \min\left( \frac{1}{2(n-1)}, \quad \frac{(1-\gamma^2) \widehat{\sigma}^2_{n-1}(\mathcal M)}{n\left(9 K_2^2 + 4\right)}\right),
$$
Then, it holds that 
\begin{align*}
 \mathbb{P}\left[ \mathcal G_n \left| \Zb^{1:\infty}\right. \right] & =  \mathbb{P}\left[ \sigb_{n-1}(\mathcal M)^2 \geq \gamma^2 \sigh_{n-1}(\mathcal M)^2  \left| \Zb^{1:\infty}\right. \right] \\
 & = \mathbb{P}\left[ \sigh_{n-1}(\mathcal M)^2 -\sigb_{n-1}(\mathcal M)^2  \leq (1- \gamma^2) \sigh_{n-1}(\mathcal M)^2  \left| \Zb^{1:\infty}\right. \right] \\
  & \geq \mathbb{P}\left[ n\left( 2 + \left(\frac{K_2 + \sqrt{n-1}\eta}{1-(n-1)\eta}\right)^2 + K_2^2\right)\eta \leq (1-\gamma^2) \widehat{\sigma}_{n-1}(\mathcal M) \left| \Zb^{1:\infty}\right. \right] \\
  & \geq \mathbb{P}\left[\eta \leq \epsilon\left| \Zb^{1:\infty}\right. \right]. \\
\end{align*}
Thus, using the definition of $\eta$ given by (\ref{eq:eta}) and applying Lemma~\ref{lem:cov}, we then obtain that
\begin{align*}
\mathbb{P}\left[\mathcal G_n |\Zb^{1:(n-1)} \right] & \geq \mathbb{P}\left[\eta \leq \epsilon\left| \Zb^{1:\infty}\right. \right] \\
 &=  \mathbb{P}\left[\left. \sup_{g,h \in \mathcal{M}^{n-1}} \left| {\rm Cov}\left[g(Z), h(Z)\left|\Zb^{1:\infty}\right.\right] - {\rm Cov}_{\overline{Z}^n}\left(g, h\right)\right| \leq \epsilon \right| \Zb^{1:(n-1)}\right]\\
 & \geq 1 -\delta_n,\\
\end{align*}
since 
$$
Ce^{-c M_n \phi\left( \kappa_{n-1}\right)} \leq \delta_n,
$$
with 
$$
\kappa_{n-1}:= \frac{\min\left( \frac{1}{2(n-1)}, \quad \frac{(1-\gamma^2) \widehat{\sigma}^2_{n-1}(\mathcal M)}{n\left(9 K_2^2 + 4\right)}\right)}{6K_\infty^{n-1}   K_{\mathcal L}^{n-1}},
$$
which yields the desired result.
\end{proof}

We are now in position to end the proof of Theorem~\ref{th:main}.
\begin{proof}[Proof of Theorem~\ref{th:main}]
Collecting Lemma~\ref{lem:first} and Lemma~\ref{lem:nit}, we obtain (\ref{eq:firstrel}) for all $n\in \mathbb{N}^*$. Let us now prove (\ref{eq:lastres}).

Let us first prove by recursion that for all $n\in \mathbb{N}^*$, 
\begin{equation}\label{eq:rec}
\mathbb{P}\left[ \bigcap_{k=1}^n \mathcal G_k \right] \geq \Pi_{k=1}^n (1-\delta_k). 
\end{equation}
Using Lemma~\ref{lem:first}, it holds that (\ref{eq:rec}) is true for $n=1$. Now we turn to the proof of the recursion.  
Let $n\in \mathbb{N}^*$. For any event $\mathcal Z$, we denote by $\un_{\mathcal Z}$ the random variable which is equal to $1$ if $\mathcal Z$ is realized and $0$ if not. 
Using the fact that $\bigcap_{k=1}^{n} \mathcal G_k$ is measurable with respect to $\overline{Z}^{1:n}$, it holds that 
\begin{align*}
 \mathbb{P}\left[ \bigcap_{k=1}^{n+1} \mathcal G_k \right] & = \mathbb{E}\left[ \un_{\bigcap_{k=1}^{n+1} \mathcal G_k}\right] \\
 & =  \mathbb{E}\left[ \mathbb{E}\left[ \un_{\mathcal G_{n+1}} \un_{\bigcap_{k=1}^{n} \mathcal G_k} \left| \overline{Z}^{1:n}\right.\right]\right] \\
 & = \mathbb{E}\left[ \mathbb{E}\left[ \un_{\mathcal G_{n+1}}  \left| \overline{Z}^{1:n}\right.\right] \un_{\bigcap_{k=1}^{n} \mathcal G_k}\right] \\
  & = \mathbb{E}\left[ \mathbb{P}\left[ \mathcal G_{n+1} \left| \overline{Z}^{1:n}\right.\right] \un_{\bigcap_{k=1}^{n} \mathcal G_k}\right]. \\
\end{align*}
Now using Lemma~\ref{lem:nit}, it holds that almost surely $\mathbb{P}\left[ \mathcal G_{n+1} \left| \overline{Z}^{1:n}\right.\right] \geq 1- \delta_{n+1}$. Hence, it holds that 
$$
 \mathbb{P}\left[ \bigcap_{k=1}^{n+1} \mathcal G_k \right] \geq (1- \delta_{n+1}) \mathbb{E}\left[\un_{\bigcap_{k=1}^{n} \mathcal G_k}\right] = (1- \delta_{n+1}) \mathbb{P}\left[\bigcap_{k=1}^{n} \mathcal G_k\right]. 
$$
The recursion is thus proved, together with (\ref{eq:rec}), which implies (\ref{eq:lastres}).

\medskip

If $\bigcap_{n\in\mathbb{N}} \mathcal G_n$ is realised, it then holds that the MC-greedy algorithm is a weak greedy algorithm with parameter $\gamma$ and norm $\|\cdot\|_{\Zb^{1:\infty}} = \sqrt{{\rm Var}\left[ \cdot | \overline{Z}^{1:\infty}\right]}$. 

\end{proof}

\section{Numerical results}\label{sec:resnum}

The aim of this section is to compare several procedures to choose the value of the sequence $(M_n)_{n\in \mathbb{N}^*}$ in the MC-greedy algorithm presented in Section~\ref{sec:presentation}. 

\subsection{Three numerical procedures}\label{sec:algos}

As mentioned in Remark~\ref{rem:main}, it is possible to design a constructive way to define a sequence of numbers of samples $(M_n)_{n\in\mathbb{N}^*}$ which 
satisfies assumptions of Theorem~\ref{th:main}, and thus which guarantees that the corresponding MC-greedy algorithm 
is a weak-greedy algorithm with high probability. Unfortunately, it appears that such a procedure cannot be used in practice to design a variance reduction method since the values of 
the sequence $(M_n)_{n\in \mathbb{N}^*}$ increases too sharply. The objective of this section is to propose a 
\itshape heuristic \normalfont procedure to choose a sequence $(M_n)_{n\in \mathbb{N}^*}$ for an MC-greedy algorithm. 
This heuristic procedure appears to yield a reduced basis approximation $\overline{f}_{\mu}$ of $f_\mu$ which provides very satisfactory results in terms of variance reduction, 
at least on the different test cases presented below. 

%

We use here the same notation as in Section~\ref{sec:motiv} and consider $M_{\rm ref}\in \mathbb{N}^*$ such that $M_{\rm ref}\gg 1$. The idea of this heuristic method is the following: 
assume that the sequence $(M_n)_{n\in\mathbb{N}^*}$ can be chosen so that for all $n\in \mathbb{N}^*$, the inequality
\begin{align} \nonumber
& \left| \mathop{\inf}_{(\lambda_i)_{1\leq i \leq n-1}\in \mathbb{R}^{n-1}} {\rm Var}\left[\left. g_{\mu}(Z) - \sum_{i=1}^{n-1} \lambda_i \gb_i(Z)\right| \Zb^{1:\infty}\right] - \mathop{\inf}_{(\lambda_i)_{1\leq i \leq n-1}\in \mathbb{R}^{n-1}} {\rm Var}_{\overline{Z}^n}\left(g_{\mu} - \sum_{i=1}^{n-1} \lambda_i \gb_i\right)\right| \\ \label{eq:try}
& \leq (1- \gamma^2)  \mathop{\inf}_{(\lambda_i)_{1\leq i \leq n-1}\in \mathbb{R}^{n-1}} {\rm Var}\left[\left. g_{\mu}(Z) - \sum_{i=1}^{n-1} \lambda_i \gb_i(Z)\right| \Zb^{1:\infty}\right] \\\nonumber
\end{align}
holds for all $\mu \in \mathcal P$. Then, it can easily be checked that such an MC-greedy algorithm is a weak greedy algorithm with parameter $\gamma$. Of course, such an algorithm could not be of any use for variance reduction since 
it would imply the computation of $\displaystyle \mathop{\inf}_{(\lambda_i)_{1\leq i \leq n-1}\in \mathbb{R}^{n-1}} {\rm Var}\left[\left. g_{\mu}(Z) - \sum_{i=1}^{n-1} \lambda_i \gb_i(Z)\right| \Zb^{1:\infty}\right]$ (or an approximation of this quantity 
of the form $\displaystyle \mathop{\inf}_{(\lambda_i)_{1\leq i \leq n-1}\in \mathbb{R}^{n-1}} {\rm Var}_{\overline{Z}^{\rm ref}}\left( g_{\mu} - \sum_{i=1}^{n-1}
\lambda_i \gb_i\right)$ with $\overline{Z}^{\rm ref} = \left( Z^{\rm ref}_k\right)_{1\leq k \leq M^{\rm ref}}$ a collection of iid random variables with the same law as $Z$ and independent of $Z$) for all $\mu \in \mathcal P$. 

\medskip

The idea of the heuristic procedure is then to check if the inequality (\ref{eq:try}) holds, \itshape only for the value $\mu = \mub_n$\normalfont. If the inequality holds, the value of $M_{n+1}$ is chosen to be equal to $M_n$ for the next iteration. 
Otherwise, the value of $M_n$ is increased and the $n^{th}$ iteration of the MC-greedy algorithm is performed again.

\medskip

This heuristic procedure leads to the Heuristic MC-greedy algorithm (or HMC-greedy algorithm), see Algorithm~\ref{MCWGHeuristic}. Notice that we introduce here 
$\mathcal{P}_{\rm trial}$ a finite subset of $\mathcal P$, which is classically called the trial set of parameters in reduced basis methods. \medskip

For the sake of comparison, we introduce two other algorithms, which cannot be implemented in practice, but which will allow us to compare the performance of the HMC-greedy algorithm with ideal procedures. 
The first method, called SHMC-greedy algorithm and also presented in Algorithm~\ref{MCWGHeuristic} as a variant, consists in designing the sequence $(M_n)_{n\in \mathbb{N}^*}$ in order to ensure that the inequality 
(\ref{eq:try}) is satisfied for all $\mu \in \mathcal P_{\rm trial}$ (and not only for $\mub_n$). 
The second one consists in performing an \itshape ideal \normalfont MC-greedy algorithm, called herefater IMC-greedy algorithm, see Algorithm~\ref{MCWGIdeal}, where all the variances and expectations are 
evaluated using $M_{\rm ref}$ number of samples of the random variable $Z$ at each iteration of the MC-greedy algorithm.

\medskip

Let us comment on the termination criterion 
$$
\frac{{\rm Var}_{\Zb^{\rm ref}}\left( \overline{f}_{\mu^{(S)H}_{n-1}} \right)}{M_{\rm ref}} > \frac{{\rm Var}_{\Zb^{\rm ref}}\left( f_{\mu^{(S)H}_{n-1}} - \overline{f}_{\mu^{(S)H}_{n-1}} \right)}{M_{n-1}}
$$
introduced in line~11 of the (S)HMC-greedy algorithm. Recall that, for $\mu = \mu^{(S)H}_{n-1}$, the expectation $\mathbb{E}\left[f_{\mu^{(S)H}_{n-1}}(Z)\right]$ is approximated after $n-1$ iterations of
the greedy algorithm by the control variate 
formula (see (\ref{eq:controlvar}))
\begin{equation}\label{eq:approx}
\mathbb{E}_{\overline{Z}^{\rm ref}}(\overline{f}_{\mu^{(S)H}_{n-1}})+ \mathbb{E}_{\Zb^{n-1}}\left(f_{\mu^{(S)H}_{n-1}}-\overline{f}_{\mu^{(S)H}_{n-1}}\right).
\end{equation}
This criterion ensures that the iterative scheme ends when the statistical error associated with the second term in (\ref{eq:approx}) becomes smaller than the statistical error of the first term (see Remark~\ref{rem:rem1}).

\begin{algorithm}

\scriptsize
\SetKwInOut{Input}{input}\SetKwInOut{Output}{output}
\Input{$\gamma >0$, $\epsilon>0$, $M_1 \in \mathbb{N}^*$, ${\mathcal P}_{\rm trial}$ trial set of parameters (finite subset of $\mathcal P$), $M_{\rm ref}\in \mathbb{N}^*$ 
(high fidelity sampling number, which has a vocation to satisfy $M_{\rm ref} \gg M_1$. 
}


\Output{$N\in \mathbb{N}^*$ size of the reduced basis, $\mu_1^{(S)H}, \mu_2^{(S)H}, \cdots, \mu_N^{(S)H} \in \mathcal{P}_{\rm trial}$, $\left(\mathbb{E}_{\overline{Z}^{\rm ref}}(f_{\mu^{(S)H}_n})\right)_{1\leq n \leq N}$.}

\BlankLine

\nl Let $\overline{Z}^{\rm ref}:=\left( Z_k^{\rm ref}\right)_{1\leq k \leq M_{\rm ref}}$ be a collection of $M_{\rm ref}$ iid random variables with the same law as $Z$ and independent of $Z$.\label{line0h}\\

\nl Set $R^1=1$.\\
 
 \nl \While{$R^1 \geq 1 - \gamma^2$ \label{Loop2}}{ 
    \nl Let $\overline{Z}^1:=\left( Z_k^1\right)_{1\leq k \leq M_1}$ be a collection of $M_1$ iid random variables with the same law as $Z$ and independent of $Z$ and $\overline{Z}^{\rm ref}$.\label{line2h}\\
    \nl Compute $\displaystyle \mu^{(S)H}_1 \in \mathop{\rm argmax}_{\mu \in {\mathcal P_{\rm trial}}} {\rm Var}_{\overline{Z}^1}\left( f_\mu \right)$.\label{line3h}\\
    \nl Compute $\overline{f}_{\mu^{(S)H}_1} = 0$.\\
    \nl Compute $\mathbb{E}_{\overline{Z}^{\rm ref}}(f_{\mu^{(S)H}_1})$.\\ 
    \nl \begin{itemize}
         \item  \bfseries HMC case: \normalfont  Set $\theta^H_1(\mu_1^H) =  \sqrt{{\rm Var}_{\overline{Z}^{\rm ref}}\left( f_{\mu^H_1}\right)}$ and $\beta^H_1(\mu_1^H) = \sqrt{{\rm Var}_{\overline{Z}^1}\left( f_{\mu^H_1}\right)}$. 
          Set $R^1= \frac{\left|\theta^H_1(\mu_1^H)^2-\beta^H_1(\mu_1^H)^2\right|}{\theta^H_1(\mu_1^H)^2}$.
          
          \item \bfseries SHMC case: \normalfont Set $\theta^{SH}_1(\mu) =  \sqrt{{\rm Var}_{\overline{Z}^{\rm ref}}\left( f_{\mu}\right)}$ and $\beta^{SH}_1(\mu) = \sqrt{{\rm Var}_{\overline{Z}^1}\left( f_{\mu}\right)}$ for all $\mu \in \mathcal P_{\rm trial}$. 
          Set $R^1= \sup_{\mu \in \mathcal P_{\rm trial}} \frac{\left|\theta^{SH}_1(\mu)^2-\beta^{SH}_1(\mu)^2\right|}{\theta^{SH}_1(\mu)^2}$.
        \end{itemize}

    \nl    \If {$R^1\geq 1 - \gamma^2$ }{ 
    Set $b_1:= 1.1$ and $M_1= \left\lceil b_1M_1 + 1\right\rceil$. \label{line5h}}
}
\nl Compute   $\gb^1 = \frac{f_{\mu^{(S)H}_1} - \mathbb{E}_{\overline{Z}^{\rm ref}}\left(f_{\mu^{(S)H}_1}\right) }{\theta^{(S)H}_1(\mu_1^{(S)H})}$. Set $n=2$ and $M_n=M_1$.\\
\nl \While{ $\frac{{\rm Var}_{\Zb^{\rm ref}}\left( \overline{f}_{\mu^{(S)H}_{n-1}} \right)}{M_{\rm ref}} \leq \frac{{\rm Var}_{\Zb^{\rm ref}}\left( f_{\mu^{(S)H}_{n-1}} - \overline{f}_{\mu^{(S)H}_{n-1}} \right)}{M_{n-1}}$}{\label{Loop1h}
  \nl \\
 \nl Set $R^n=1$.\\
  \nl \While{$R^n \geq 1 - \gamma^2$ \label{Loop2h}}{ 
     \nl Let $\overline{Z}^n:=\left( Z_k^n\right)_{1\leq k \leq M_n}$ be a collection of $M_n$ iid random variables with the same law as $Z$ and independent of $Z$ and $\overline{Z}^{\rm ref}$.\label{line8h}\\
    \nl Compute $\displaystyle \mu^{(S)H}_n \in \mathop{\rm argmax}_{\mu \in {\mathcal P_{\rm trial}} } \mathop{\rm min}_{(\lambda_i)_{1\leq i \leq n-1}\in \R^{n-1} } {\rm Var}_{\overline{Z}^n}\left( f_\mu - \sum_{i=1}^{n-1} \lambda_i \gb_{i}\right)$\label{line10h}\\
    \nl Compute $(\overline{\lambda}^n_i)_{1\leq i \leq n-1}= \mathop{\rm argmin}_{(\lambda_i)_{1\leq i \leq n-1}\in \R^{n-1} } \sqrt{{\rm Var}_{\overline{Z}^{\rm ref}}\left( f_{\mu^{(S)H}_n} - \sum_{i=1}^{n-1}  \lambda_i \gb_{i}\right)}$.\\
      \nl Compute $\overline{f}_{\mu^{(S)H}_n}= \sum_{i=1}^{n-1}  \overline{\lambda}^n_i \gb_{i}$.\\

    \begin{itemize}
          \item \bfseries HMC case: \normalfont Compute $\displaystyle \theta_n^H(\mu^H_n) = \mathop{\rm min}_{(\lambda_i)_{1\leq i \leq n-1}\in \R^{n-1} } \sqrt{{\rm Var}_{\overline{Z}^{\rm ref}}\left( f_{\mu^{(S)H}_n} - \sum_{i=1}^{n-1}  \lambda_i \gb_{i}\right)} = \sqrt{{\rm Var}_{\overline{Z}^{\rm ref}}\left( f_{\mu^{(S)H}_n} - \overline{f}_{\mu^{(S)H}_n}\right)}$ 
          and $\displaystyle \beta_n^H(\mu^H_n) = \mathop{\min}_{(\lambda_i)_{1\leq i \leq n-1}\in \R^{n-1} } \sqrt{{\rm Var}_{\overline{Z}^{n}}\left( f_{\mu^H_n} - \sum_{i=1}^{n-1}  \lambda_i \gb_{i}\right)}$. Set $\displaystyle R^n= \left|\frac{\theta_n^H(\mu_n^H)^2-\beta_n^H(\mu_n^H)^2}{\theta_n^H(\mu_n^H)^2}\right|.$ 
           \item \bfseries SHMC case: \normalfont For all $\mu \in \mathcal P_{\rm trial}$, compute $\displaystyle \theta_n^{SH}(\mu) = \mathop{\rm min}_{(\lambda_i)_{1\leq i \leq n-1}\in \R^{n-1} } \sqrt{{\rm Var}_{\overline{Z}^{\rm ref}}\left( f_{\mu} - \sum_{i=1}^{n-1}  \lambda_i \gb_{i}\right)}$ 
          and $\displaystyle \beta_n^{SH}(\mu) = \mathop{\min}_{(\lambda_i)_{1\leq i \leq n-1}\in \R^{n-1} } \sqrt{{\rm Var}_{\overline{Z}^{n}}\left( f_{\mu} - \sum_{i=1}^{n-1}  \lambda_i \gb_{i}\right)}$.
          Set $\displaystyle R^n= \mathop{\sup}_{\mu \in \mathcal P_{\rm trial}}\left|\frac{\theta_n^{SH}(\mu)^2-\beta_n^{SH}(\mu)^2}{\theta_n^{SH}(\mu)^2}\right|.$ 
        \end{itemize}

        \nl \If {$R^n \geq 1 - \gamma^2$ }{
        Compute $r_n:= \frac{\phi(\theta_{n-1}^{(S)H}(\mu_{n-1}^{(S)H})^2 )}{\phi\left( \theta_n^{(S)H}(\mu_n^{(S)H})^2 \right)}$, $b_n:= \max(1.1, r_n)$ and set $M_n=\left\lceil b_n M_n +1\right\rceil$.}\label{line14H}
  }
    \nl Compute 
     $
     \displaystyle \gb_n = \frac{f_{\mu^{(S)H}_n} - \sum_{i=1}^{n-1}\lambb^n_i \gb_{i}}{\theta_n^{(S)H}(\mu_n^{(S)H})} = \frac{f_{\mu^{(S)H}_n} - \overline{f}_{\mu^{(S)H}_n}}{\theta_n^{(S)H}(\mu_n^{(S)H})}
     $
     and $\mathbb{E}_{\overline{Z}^{\rm ref}}(f_{\mu^{(S)H}_n})$.\\
     \nl Set $M_{n+1} = M_n$ and $n = n+1$.\\
    }
    Set $N=n$, $M_N=M_n$.\label{line17h}

\caption{(S)HMC-greedy algorithm}\label{MCWGHeuristic}
\end{algorithm}\DecMargin{1em}

\begin{algorithm}
\SetKwInOut{Input}{input}\SetKwInOut{Output}{output}
\Input{$\epsilon>0$, ${\mathcal P}_{\rm trial}$ trial set of parameters (finite subset of $\mathcal P$), $M_{\rm ref}\in \mathbb{N}^*$ 
(high fidelity sampling number). 
}


\Output{$N\in \mathbb{N}^*$ size of the reduced basis, $\mu^I_1, \mu^I_2, \cdots, \mu^I_N \in \mathcal{P}_{\rm trial}$, $\left(\mathbb{E}_{\overline{Z}^{\rm ref}}(f_{\mu^I_n})\right)_{1\leq n \leq N}$.}

\BlankLine

\nl Let $\overline{Z}^{\rm ref}:=\left( Z_k^{\rm ref}\right)_{1\leq k \leq M_{\rm ref}}$ be a collection of $M_{\rm ref}$ iid random variables with the same law as $Z$ and independent of $Z$.\label{line0h}\\

  \nl Compute $\displaystyle \mu^I_1 \in \mathop{\rm argmax}_{\mu \in {\mathcal P_{\rm trial}}} {\rm Var}_{\overline{Z}^{\rm ref}}\left( f_\mu \right)$.\label{line3h} \\
\nl Set $\theta^I_1(\mu_1^I):= {\rm Var}_{\overline{Z}^{\rm ref}}\left( f_{\mu^I_1} \right)$ and compute $\mathbb{E}_{\overline{Z}^{\rm ref}}\left( f_{\mu^I_1}\right)$. \\

\nl Set $n=2$, $M_n=M_1$, $\gb^1 = \frac{f_{\mu^I_1}}{\sqrt{{\rm Var}_{\overline{Z}^{\rm ref}}\left( f_{\mu^I_1}\right)}}$\\
\nl \While{ $\theta^I_{n-1}(\mu_{n-1}^I) \geq \epsilon$}{\label{Loop1h}
    \nl Compute $\displaystyle \mu^I_n \in \mathop{\rm argmax}_{\mu \in {\mathcal P_{\rm trial}} } \mathop{\rm min}_{(\lambda_i)_{1\leq i \leq n-1}\in \R^{n-1} } {\rm Var}_{\overline{Z}^{\rm ref}}\left( f_\mu - \sum_{i=1}^{n-1} \lambda_i \gb_{i}\right)$\label{line10h}\\
    \nl Compute $\theta^I_n(\mu_n^I) = \mathop{\rm min}_{(\lambda_i)_{1\leq i \leq n-1}\in \R^{n-1} } \sqrt{{\rm Var}_{\overline{Z}^{\rm ref}}\left( f_{\mu^I_n} - \sum_{i=1}^{n-1}  \lambda_i \gb_{i}\right)}$\\
    \nl Compute 
    $
    \displaystyle (\lambb^n_i)_{1\leq i \leq n-1} = \mathop{\rm argmin}_{(\lambda_i)_{1\leq i \leq n-1}\in \R^{n-1}}{\rm Var}_{\overline{Z}^{\rm ref}}\left( f_{\mu^I_n} - \sum_{i=1}^{n-1}\lambda_i \gb_{i}\right)
    $\\
    \nl Compute 
    $
    \displaystyle \gb_n = \frac{f_{\mu^I_n} - \sum_{i=1}^{n-1}\lambb^n_i \gb_{i}}{\theta^I_n(\mu_n^I)}
    $
    and $\mathbb{E}_{\overline{Z}^{\rm ref}}(f_{\mu^I_n})$. Set $n = n+1$.\\
    }
   
    Set $N=n$.\label{line17h}

\caption{IMC-greedy algorithm}\label{MCWGIdeal}
\end{algorithm}\DecMargin{1em}

\subsection{Definitions of quantities of interest}

For each of the test cases presented below, we plot different quantities of interest which we define here. 

For all $n\in \mathbb{N}^*$, we denote by $\mu_1^H, \cdots, \mu_n^H$ (respectively $\mu_1^{SH}, \cdots, \mu_n^{SH}$ and $\mu_1^I, \cdots, \mu_n^I$) the set of parameter values 
selected after $n$ iterations of the HMC-greedy (respectively SHMC-greedy and IMC-greedy) algorithm. We also denote by 
$\overline{V}_n^H:= {\rm Span}\left\{ g_{\mu_1^H}, \cdots, g_{\mu_n^H}\right\}$, $\overline{V}_n^{SH}:= {\rm Span}\left\{ g_{\mu_1^{SH}}, \cdots, g_{\mu_n^{SH}}\right\}$ and 
$\overline{V}_n^I:= {\rm Span}\left\{ g_{\mu_1^I}, \cdots, g_{\mu_n^I}\right\}$.

\medskip

For all $\mu \in \mathcal P$ and $n\in \mathbb{N}^*$, we define for the three algorithms presented in Section~\ref{sec:algos}, 
\begin{equation}\label{eq:thetamu}
\theta_n(\mu) := \mathop{\inf}_{(\lambda_i)_{1\leq i \leq n-1} \in \R^{n-1}}  \sqrt{{\rm Var}_{\overline{Z}^{\rm ref}}\left(g_\mu - \sum_{i=1}^{n-1} \lambda_i \gb_{i}\right)}
\end{equation}
and 
\begin{equation}\label{eq:theta}
\theta_n:= \mathop{\sup}_{\mu \in \mathcal P_{\rm trial}} \theta_n(\mu). 
\end{equation}
In what follows, we denote by $\theta^{H}_n(\mu)$ and $\theta^H_n$ (respectively by $\theta^{SH}_n(\mu)$, $\theta^{SH}_n$, $\theta^I_n(\mu)$ and $\theta^I_n$) the quantities defined by (\ref{eq:thetamu}) and (\ref{eq:theta}) 
obtained with the HMC-greedy (respectively the SHMC-greedy and IMC-greedy) algorithm.  Note that, by definition of the IMC-greedy algorithm, $\theta^I_n = \theta^I_n(\mu_n^I)$.

\medskip

A second quantity of interest for the HMC-greedy and the SHMC-greedy algorithms is given, for all $n\in \mathbb{N}^*$ and $\mu \in \mathcal P$, by
\begin{equation}\label{eq:betamu}
\beta_n(\mu):= \mathop{\inf}_{(\lambda_i)_{1\leq i \leq n-1} \in \R^{n-1}} \sqrt{ {\rm Var}_{\overline{Z}^n}\left(g_\mu - \sum_{i=1}^{n-1} \lambda_i \gb_{i}\right)}.
\end{equation}
and
\begin{equation}\label{eq:beta}
\beta_n:= \mathop{\sup}_{\mu \in \mathcal P_{\rm trial}} \beta_n(\mu).
\end{equation}
In the sequel, we denote by $\beta^{H}_n$ (respectively by $\beta^{SH}_n$) the quantity defined by (\ref{eq:beta}) obtained with the HMC-greedy (respectively the SHMC-greedy) algorithm.

Let us point out that when $\mathcal P_{\rm trial}= \mathcal P$ and when $M_{\rm ref} = \infty$, it holds that $\theta_n = \widehat{\sigma}_{n-1}(\mathcal M)$ and 
$\theta_n(\mub_n) = \overline{\sigma}_{n-1}(\mathcal M)$ where $\widehat{\sigma}_{n-1}(\mathcal M)$ and $\overline{\sigma}_{n-1}(\mathcal M)$ are defined respectively in (\ref{eq:sigh}) and (\ref{eq:sigb}).  

\medskip

We finally wish to evaluate the error made on the approximation of $\mathbb{E}[f_\mu(Z)]$ obtained by using the variance reduction method based on these MC-greedy algorithm. More precisely, this approximation is computed as
$$
 \sum_{i=1}^{n-1} \lambda_i^{n,\mu} \mathbb{E}_{\overline{Z}^{\rm ref}}(f_{\mub_i}) + \mathbb{E}_{\overline{Z}^{n}}\left(f_\mu-\sum_{i=1}^{n-1} \lambda_i^{n,\mu} f_{\mub_i}\right)
$$
where 
$$
\left( \lambda_i^{n,\mu} \right)_{1\leq i \leq n-1} = \mathop{\rm argmin}_{\left( \lambda_i\right)_{1\leq i \leq n-1}\in \mathbb{R}^{n-1}} {\rm Var}_{\overline{Z}^n}\left(f_\mu - \sum_{i=1}^{n-1}\lambda_i f_{\mub_i} \right)
$$
and $\mub_i$ is equal to $\mu_i^H$, $\mu_i^{SH}$ or $\mu_i^I$ depending on the chosen algorithm (remember formula~(\ref{eq:controlvar})). This quantity has to be compared with 
the approximation obtained with a standard Monte-Carlo with $M_{\rm ref}$ samples, i.e. $\mathbb{E}_{\overline{Z}^{\rm ref}}\left( f_\mu\right)$. 
To this aim, for all $n\in \mathbb{N}^*$ and $\mu \in \mathcal P$, 
we define
\begin{equation}\label{eq:err}
e_n(\mu):= \frac{\left|\mathbb{E}_{\overline{Z}^{\rm ref}}\left( f_\mu\right) - \left[ \sum_{i=1}^{n-1} \lambda_i^{n,\mu} \mathbb{E}_{\overline{Z}^{\rm ref}}(f_{\mub_i}) + \mathbb{E}_{\overline{Z}^{n}}\left(f_\mu-\sum_{i=1}^{n-1} \lambda_i^{n,\mu} f_{\mub_i}\right)\right]\right|}{\left| \mathbb{E}_{\overline{Z}^{\rm ref}}\left( f_\mu\right)\right|}.
\end{equation}
In what follows, we denote by $e_n^H(\mu)$ (respectively $e_n^{SH}(\mu)$) the quantity defined by (\ref{eq:err}) obtained by the HMC-greedy (respectively the SHMC-greedy) algorithm.

\subsection{Explicit one-dimensional functions}

We consider in this section two sets of one-dimensional explicit functions. The motivation for considering these two simple test cases is that the decays of the Kolmogorov $n$-widths of the associated sets $\mathcal M$ are known. 

\medskip

\subsubsection{First test case}

Let $f: \mathbb{R} \to \mathbb{R}$ be the function defined such that
\begin{equation}
\forall x\in \mathbb{R},\quad  f(x):=\left\{
      \begin{aligned}
        2\,x& \quad\text{if} \quad 0\leq x \leq 0.5,\\
        1\,&   \quad\text{if} \quad 0.5\leq x \leq 1.5,\\
        4-2\,x& \quad\text{if} \quad 1.5\leq x \leq 2,\\
        0 & \quad \text{otherwise}.\\
      \end{aligned}
    \right.
\end{equation}
Let $\cP=[0,3]$ be the set of parameter values. We consider in this first test case the family of functions $(f_\mu)_{\mu \in \mathcal P}$ such that $f_\mu(x)=f(x-\mu)$ for all $\mu$ in $\cP$ and $x\in \mathbb{R}$. 
Let $Z$ be a real-valued random variable with probability measure $\nu=\mathcal{U}(0,5)$. In this case, it is known~\cite{devore1998nonlinear} that there exists a constant $c>0$ such that $d_n(\mathcal M) \geq c  n^{-1/2}$ for all $n\in \mathbb{N}^*$. 

\medskip

In this example, $M_1=10$, $M_{\rm ref}=10^5$, $\gamma=0.9$ and the trial set $\mathcal P_{\rm trial}$ is chosen to be a set of $300$ random parameter values which were uniformly sampled in $\mathcal P$.

 Figure~\ref{fig:figM1} illustrates the evolution of the values of $M_n$ as a function of $n$ for the HMC and SHMC algorithms. 

 \begin{figure}[h]
 \begin{center}
    \includegraphics[width=12cm]{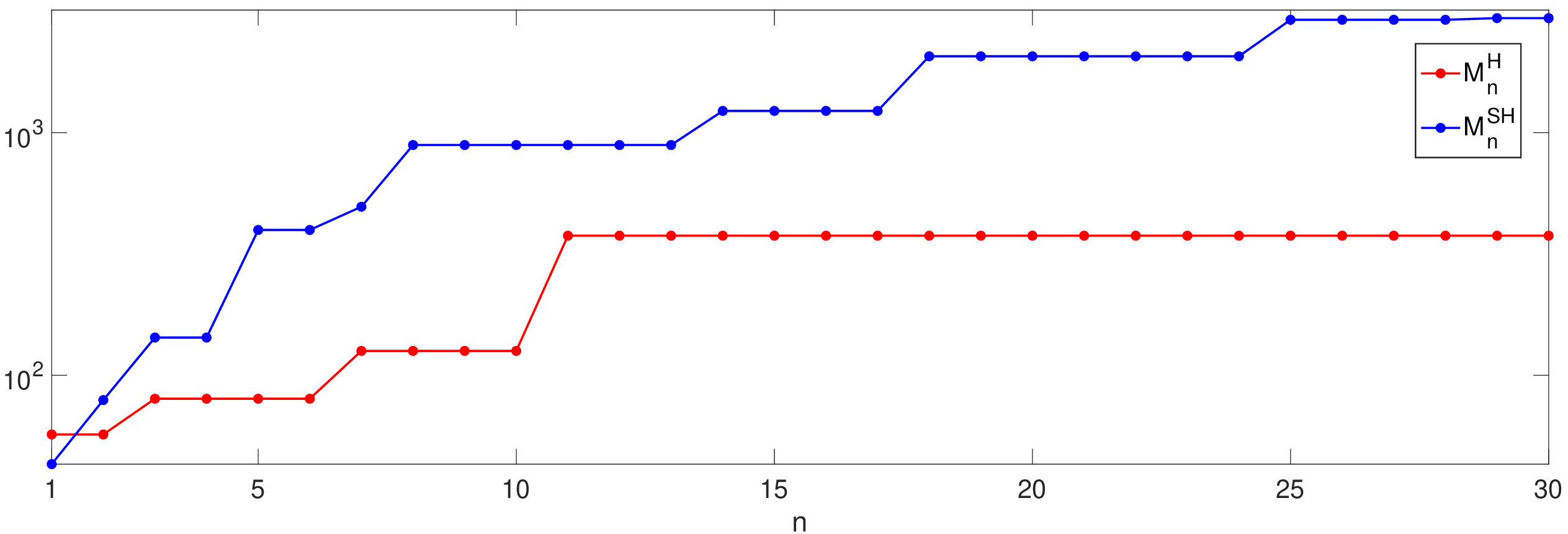}
    \caption
    {Evolution of $M_n$ as a function of $n$ for the HMC and SHMC-greedy algorithms in test case 1.} \label{fig:figM1}
\end{center}
    \end{figure}
    
    \begin{figure}[H]
\centering
   \includegraphics*[width = 12cm]{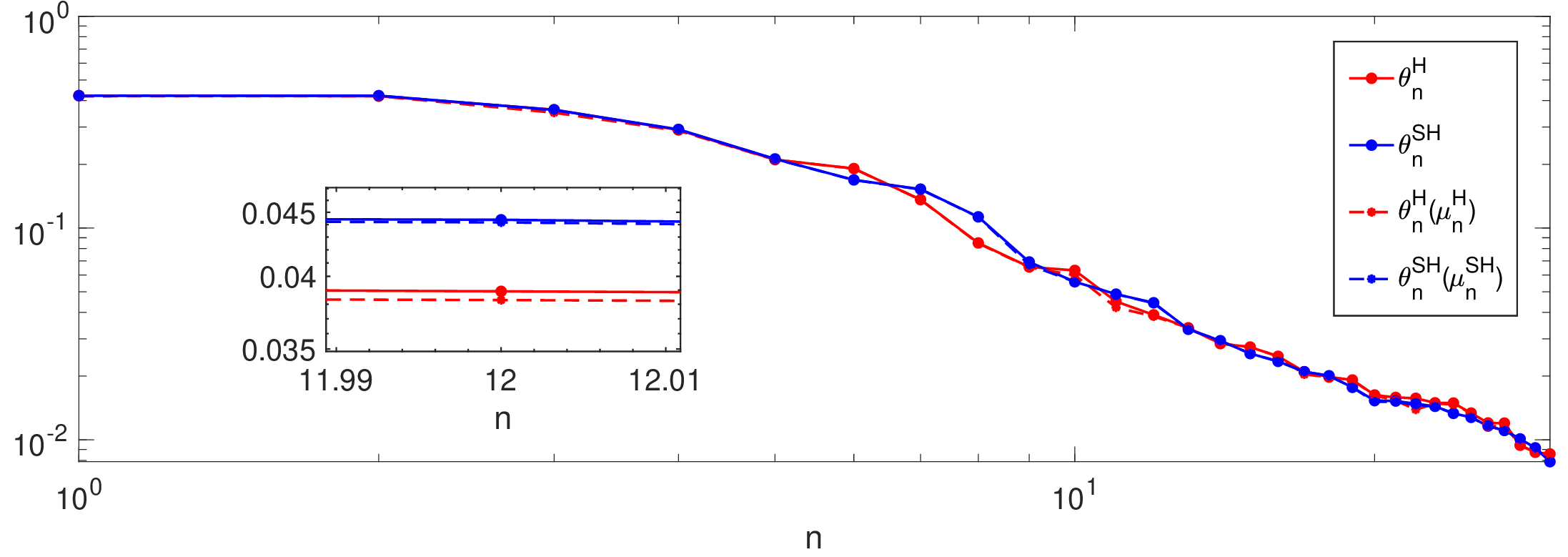}
   \includegraphics*[width=12cm]{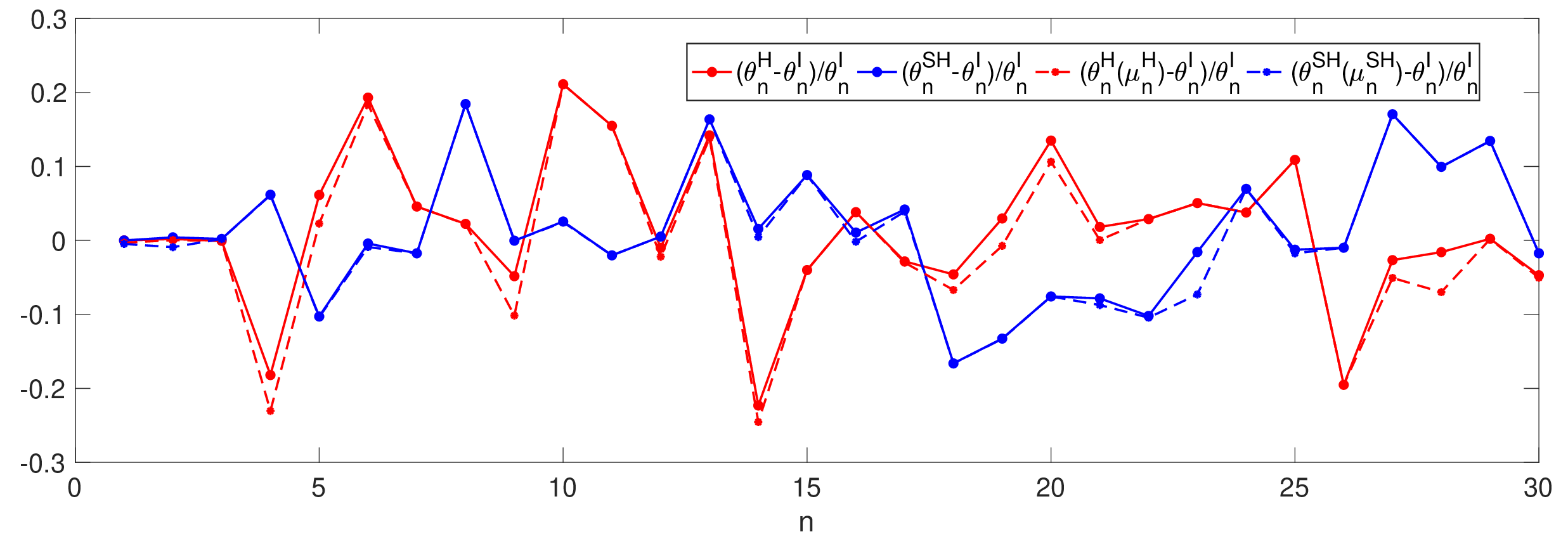}
     \caption{Evolution of $\theta^H_n(\mu_n^H)$, $\theta^{SH}_n(\mu_n^{SH})$, $\theta^H_n$, $\theta^{SH}_n$ as a function of $n$ in test case 1. }\label{Offline3Fex1}
\end{figure}

 Figure~\ref{Offline3Fex1} illustrates the fact that at each iteration $n\in \mathbb{N}^*$, for the (S)HMC-algorithm, the value of the selected parameter $\mu_n^{(S)H}$ is relevant since we observe numerically that 
 $\theta^{(S)H}(\mu_n^{(S)H})$ is very close to $\theta_n^{(S)H} = \sup_{\mu \in \mathcal P_{\rm trial}} \theta_n^{(S)H}(\mu)$. In addition, we observe that the resulting reduced spaces $\overline{V}_n^{(S)H}$ have 
 very good approximability properties 
 with respect to the set $\mathcal M$, in the sense that the values of $\theta_n^{(S)H}$ and $\theta_n^{(S)H}(\mu_n^{(S)H})$ are very close to $\theta_n^I$, which is computed with the reference IMC algorithm.

\begin{figure}[H]
\centering
   \includegraphics*[width =12cm]{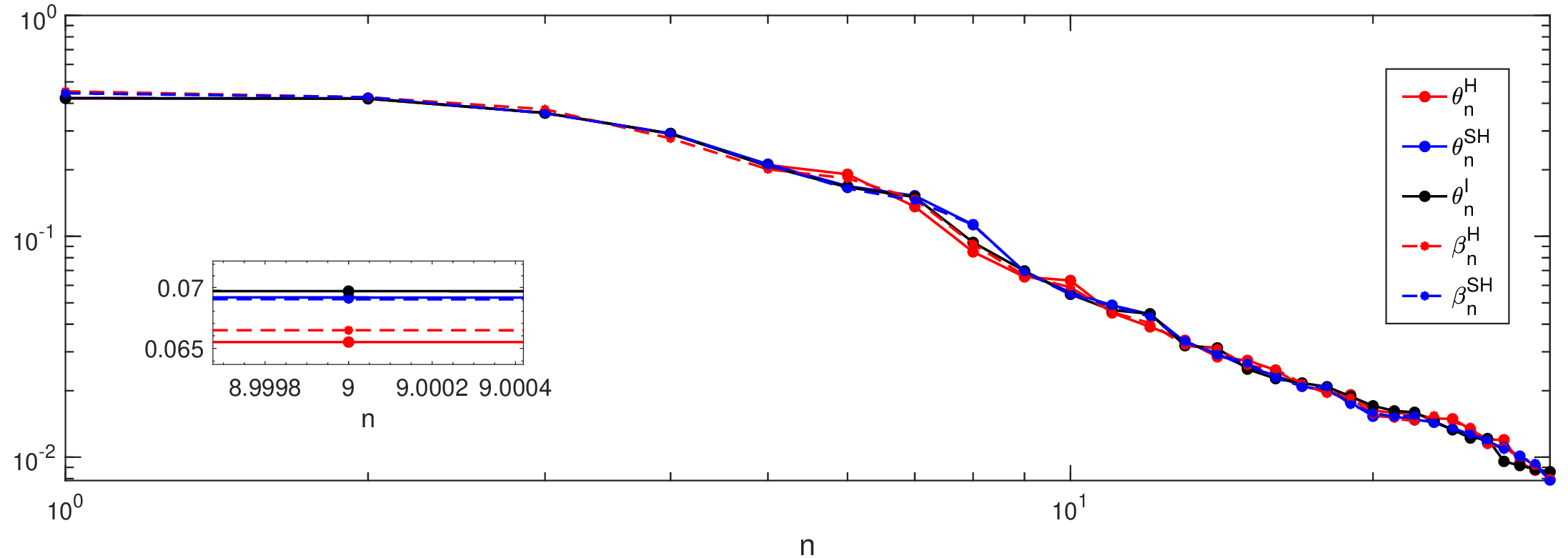}
   \includegraphics*[width=12cm]{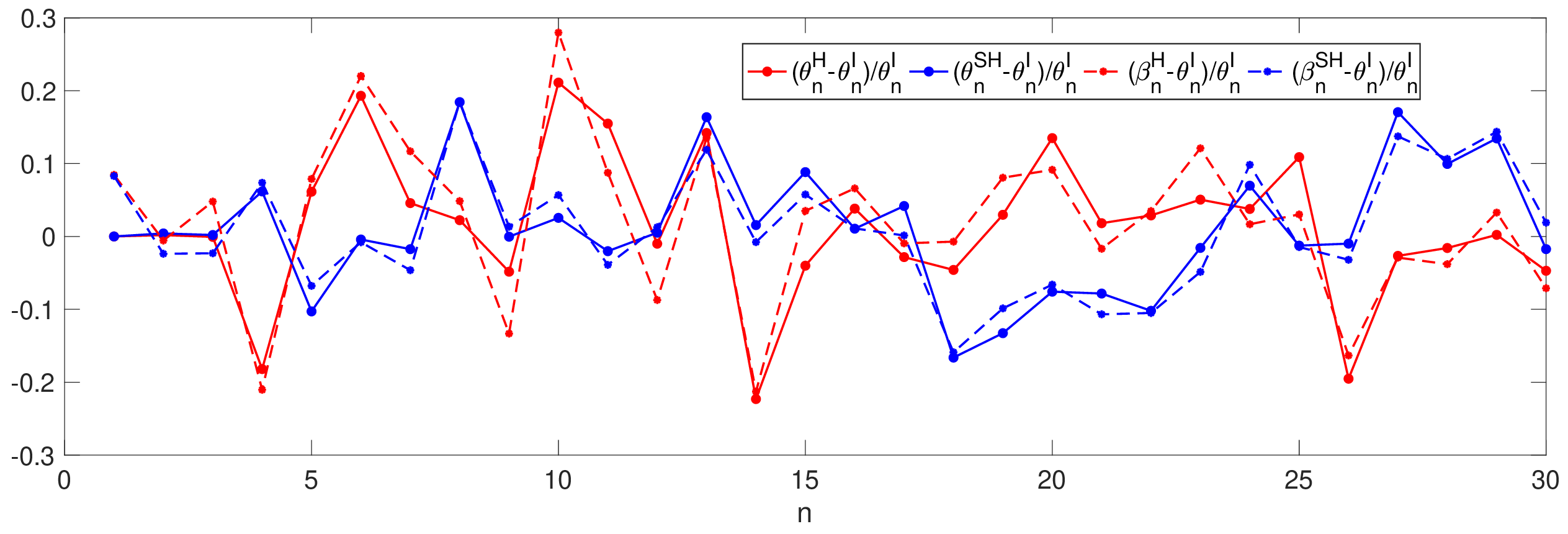}
     \caption{Evolution of $\beta^H_n$, $\beta^{SH}_n$, $\theta^H_n$, $\theta^{SH}_n$, $\theta^I_n$ as a function of $n$ in test case 1.}\label{Offline1Fex1}
\end{figure}

 Figure~\ref{Offline1Fex1} illustrates the fact that the value of the number of samples $M_n$ chosen at each iteration $n\in\mathbb{N}^*$ enables to compute empirical variances that are close to exact variances since 
 the values of $\beta_n^{(S)H}$ are very close to the $\theta_n^{(S)H}$ for the (S)HMC-algorithm. 
 
\begin{figure}[H]
\centering
   \includegraphics[width=12cm]{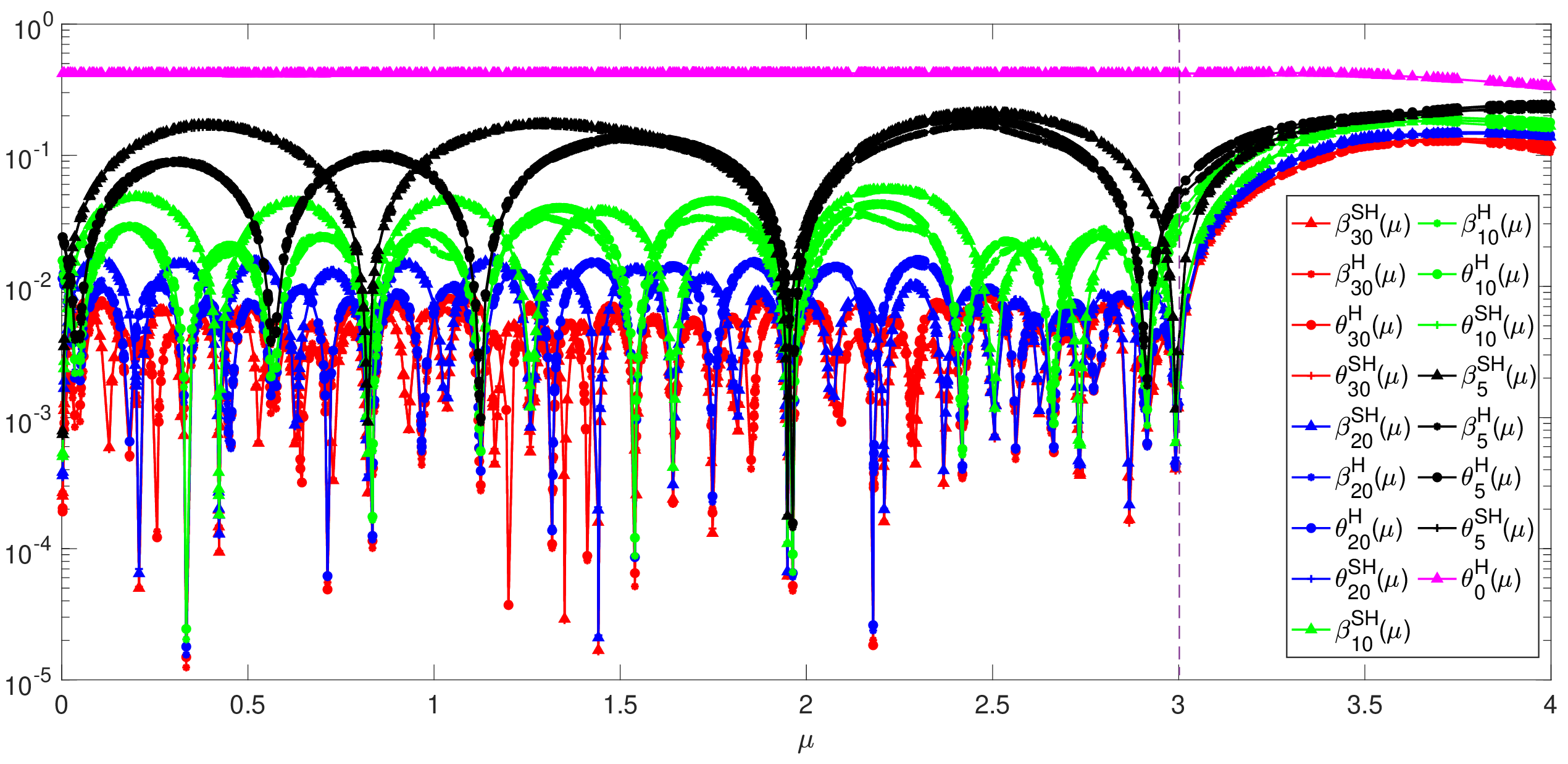}
     \caption{$\theta^H_n(\mu)$, $\theta_n^{SH}(\mu)$, $\beta_n^H(\mu)$, $\beta_n^{SH}(\mu)$ as a function of $\mu$ for $n=0, 5,10,20,30$ on $\cP_{test}=[0,4]$.}\label{Online1Fex1}
\end{figure}

In Figure~\ref{Online1Fex1}, the values of $\theta^{(S)H}_n(\mu)$ and $\beta_n^{(S)H}(\mu)$ are plotted as a function of $\mu \in \cP_{test}=[0,4]$ for different values of $n$ ($n=0, 5, 10,20,30$).

%

\begin{figure}[H]
\centering
   \includegraphics[width=12cm]{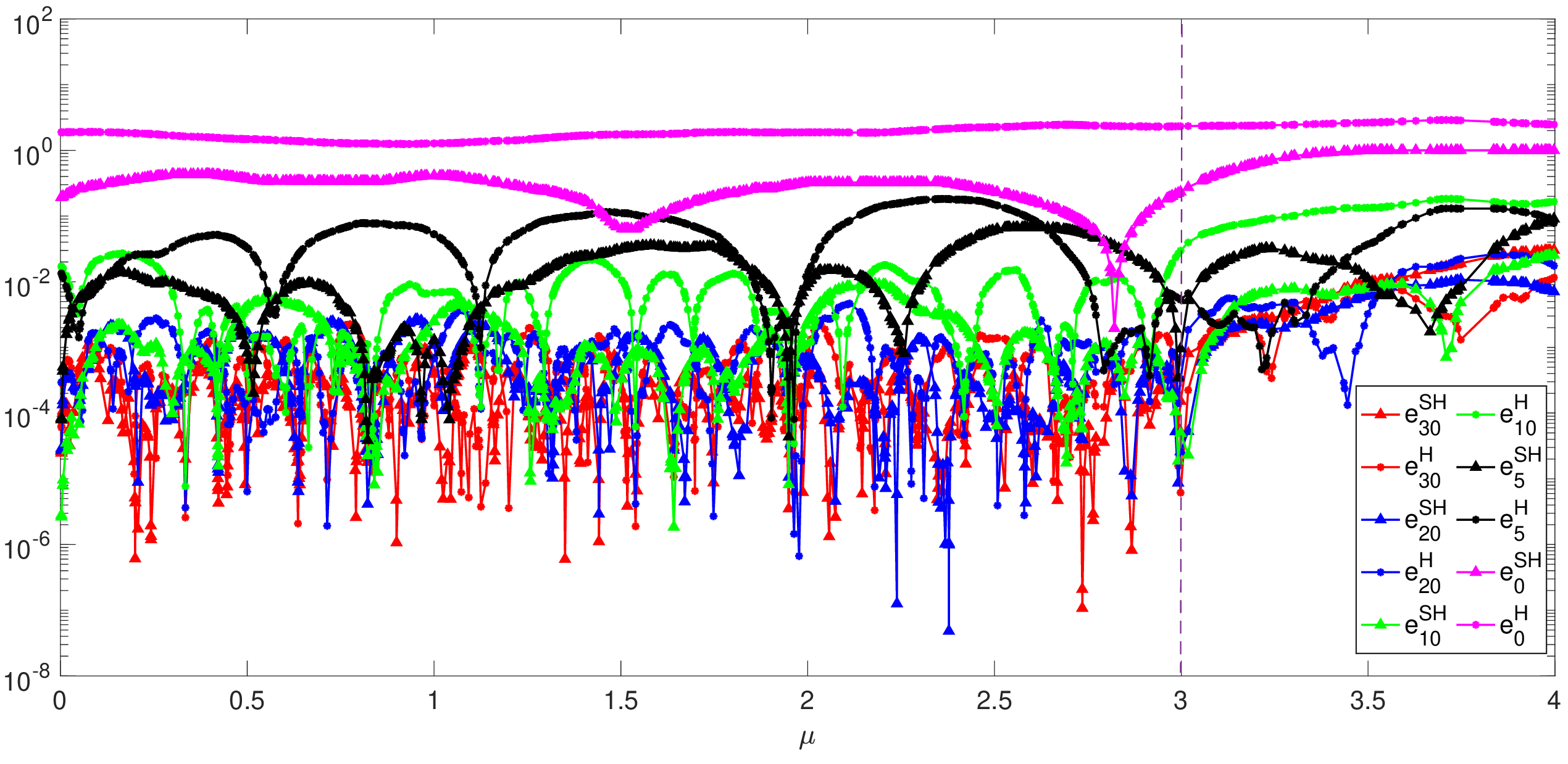}
     \caption{$e_n^H(\mu)$ and $e_n^{SH}(\mu)$ as a function of $\mu$ for $n=0,5,10,20,30$ on $\cP_{test}=[0,4]$.}\label{Online2Fex1}
\end{figure}

In comparison, in Figure~\ref{Online2Fex1}, the relative error $e_n^{(S)H}(\mu)$ is plotted as a function of $\mu$ for $n=0, 5, 10,20,30$. In particular, we observe that this error 
remains lower than $1\%$ as soon as $n \geq 10$ on $\mathcal P$. Naturally, this error is larger for $\mu \in \cP_{test}\setminus \cP$.

\begin{figure}[H]
\centering
   \includegraphics[width=12cm]{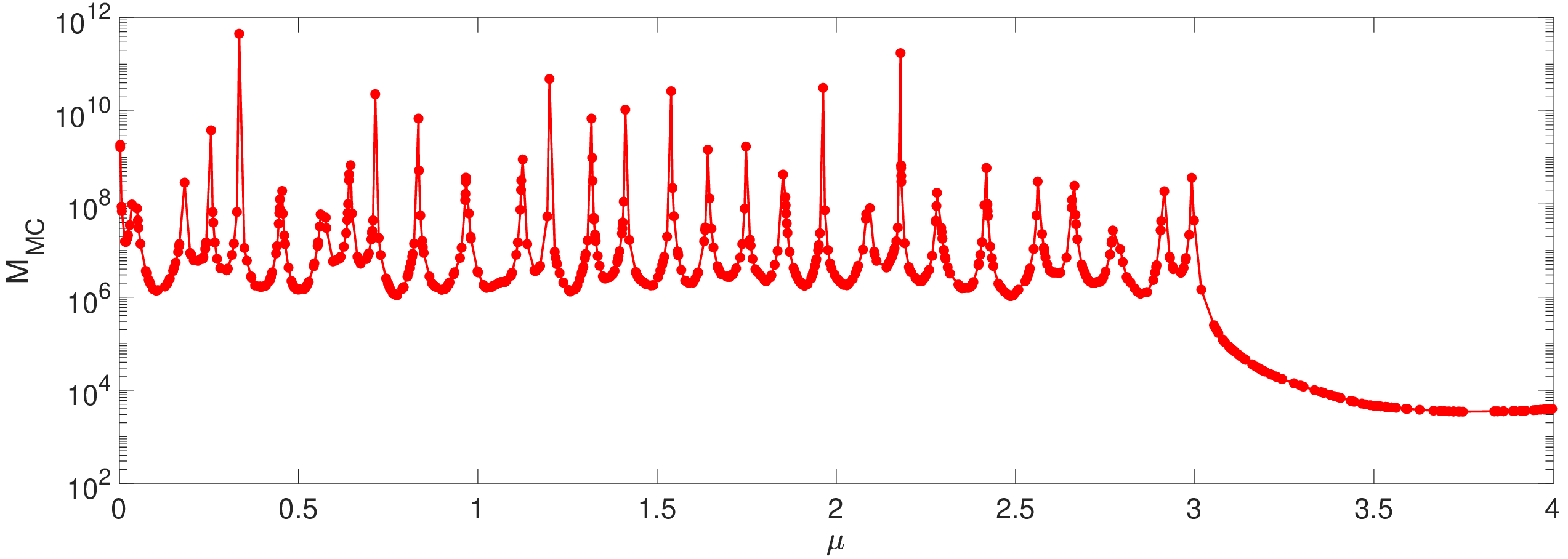}
     \caption{$M_{MC}(\mu)$ as a function of $\mu \in \cP_{test}=[0,4]$.  }\label{MrefGainfex1}
\end{figure}
Finally, to illustrate the gain of our proposed method in terms of variance reduction, we plot on Figure~\ref{MrefGainfex1} the value of the number of random Monte-Carlo samples $M_{MC}(\mu)$ that would have 
been necessary to compute an approximation of the mean of $f_\mu(Z)$ with a standard Monte-Carlo method with the same level of accuracy than the one given by the HMC-algorithm after $n=30$ iterations. 
In this case, let us point out that $M_n= 349$. More precisely, we compute $M_{MC}(\mu)$ by the follwoing formula: 
\begin{equation}\label{eq:MMC}
M_{MC}(\mu) = \frac{{\rm Var}_{\overline{Z}^{ref}}\left(f_{\mu}  \right)\times M_n}{{\rm Var}_{\overline{Z}^n}\left(f_{\mu} - \sum_{i=1}^{n}\lambda^{\mu}_i f_{\mub_i} \right)}. 
\end{equation}
Figure~\ref{MrefGainfex1} illustrates that, for all $\mu \in \mathcal P$, the classical Monte Carlo method would have required a number of samples $M_{MC}(\mu)$ in the range  
$10^6\leq M_{MC}(\mu)\leq 10^{12}$ in order to obtain the same level of statistical error. Thus, we see that the HMC-algorithm significantly improves the efficiency of the computation of the expectation of $f_\mu(Z)$ 
with respect to a standard Monte-Carlo algorithm.

\subsubsection{Second test case}

In this example, we consider a second family of one-dimensional functions where $\cP=[0,1]$ is the set of parameter values. 
More precisely, we consider the family of functions $(f_\mu)_{\mu \in \mathcal P}$ such that, for all $\mu$ in $\cP$:
\begin{equation}
\forall x\in [0,1],\quad   f_\mu(x):=\left\{
      \begin{aligned}
        \sqrt{x+0.1} \quad\text{if} \quad x \in [0, \mu]\\
        \frac{1}{2}(\mu+0.1)^{-\frac{1}{2}}x- \frac{1}{2}(\mu+0.1)^{-\frac{1}{2}}\mu +(\mu+0.1)^{\frac{1}{2}} \quad\text{if} \quad x \in [\mu, 1]\\
      \end{aligned}
    \right.
\end{equation}

Let us point out that for all $\mu \in \mathcal P$, $f_\mu$ is a $\mathcal C^1$ function on $[0,1]$. In this case, it is known~\cite{ehrlacher2020nonlinear} that there exists a constant $c>0$ such that $d_n(\mathcal M) \leq c  n^{-2}$ 
for all $n\in \mathbb{N}^*$. 

\medskip

Let $Z$ be a random variable with probability measure $\nu=\mathcal{U}(0,1)$.

\medskip

In this example, $M_1=10$, $M_{\rm ref}=10^5$, $\gamma=0.9$ and the trial set $\mathcal P_{\rm trial}$ is chosen to be a set of $300$ random parameter values which were uniformly sampled in $\mathcal P$. 
In this test case, we osbserve a similar behaviour of the (S)HMC-algorithm as in the first test case.

Figure~\ref{MrefGainfex2} illustrates the computational gain brought by the HMC algorithm after $n=70$ iterations (so that $M_n = 3109$)  with respect to the classical Monte Carlo method. 
Indeed, the quantity $M_{MC}(\mu)$ defined in (\ref{eq:MMC}) is observed to vary in this case between $10^{12}$ and $10^{18}$.

 \begin{figure}[h]
 \begin{center}
    \includegraphics[width=12cm]{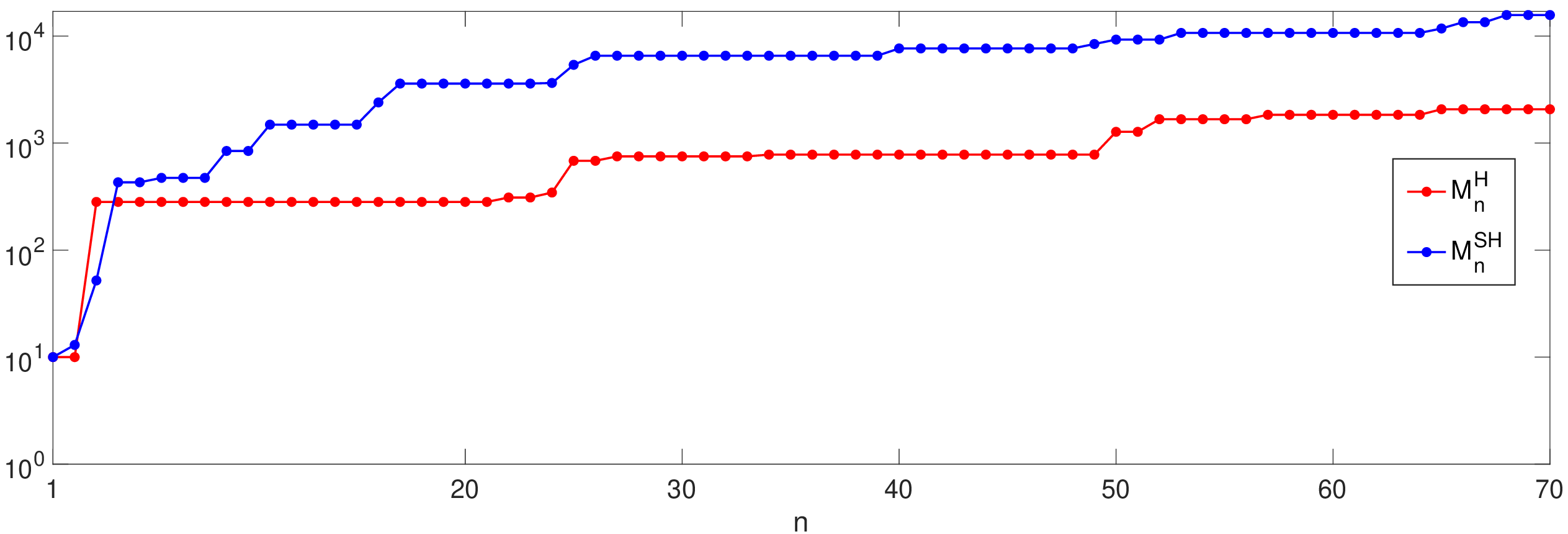}
    \caption
    {Evolution of $M_n$ as a function of $n$ for the HMC and SHMC-greedy algorithms.} \label{MnSHMnHfex}
\end{center}
    \end{figure}

    \begin{figure}[H]
\centering
   \includegraphics*[width = 12cm]{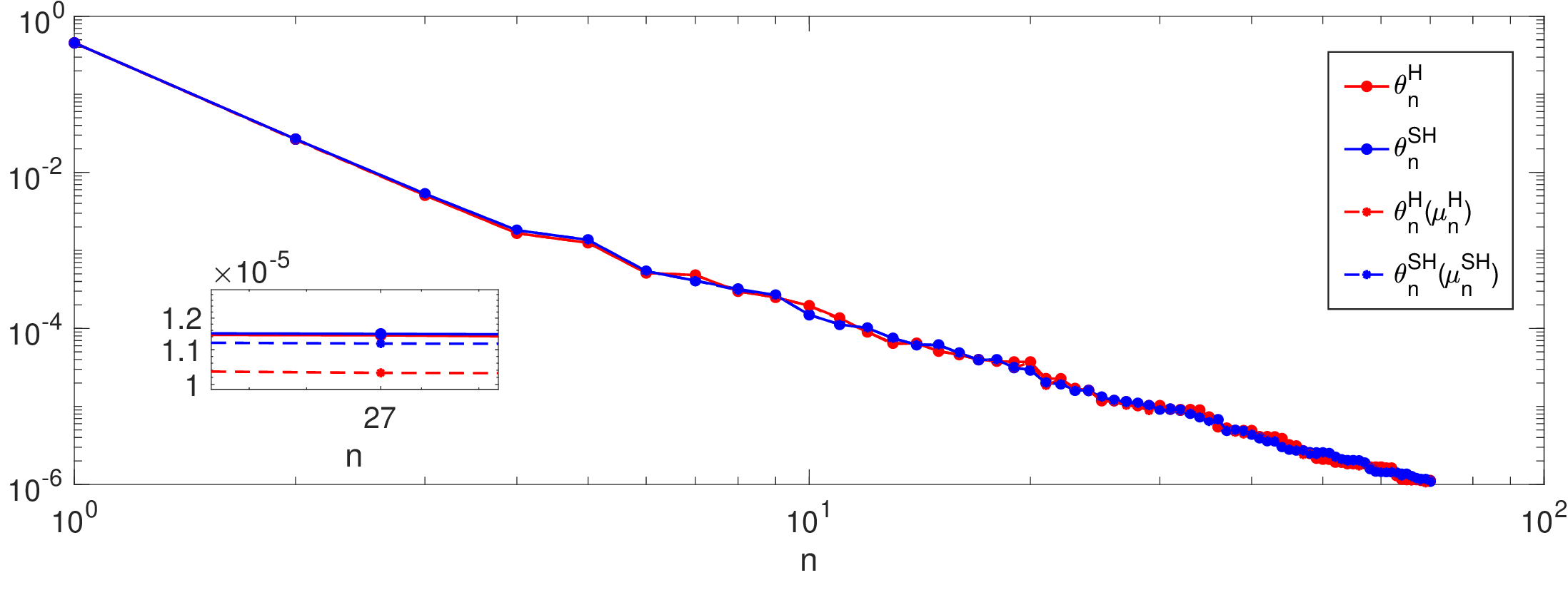}
     \caption{Evolution of $\theta^H_n(\mu_n^H)$, $\theta^{SH}_n(\mu_n^{SH})$, $\theta^H_n$, $\theta^{SH}_n$ as a function of $n$ in test case 2. }\label{offline3Fex2}
\end{figure}

\begin{figure}[H]
\centering
   \includegraphics*[width =12cm]{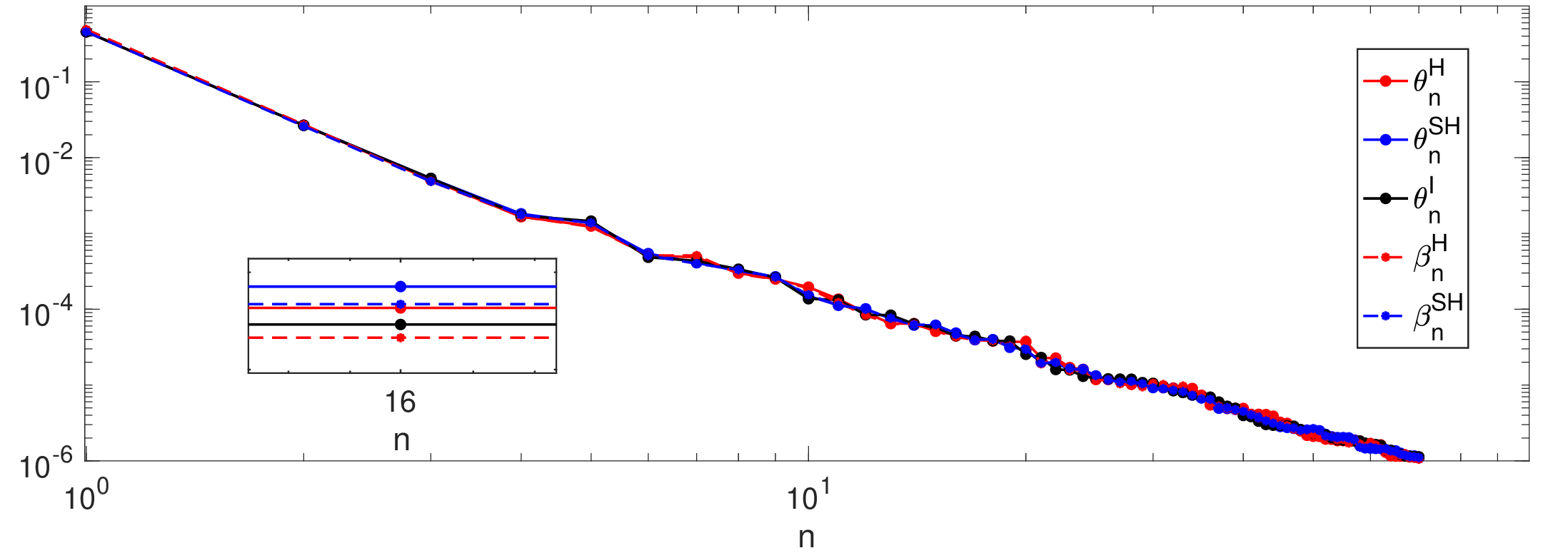}
     \caption{Evolution of $\beta^H_n$, $\beta^{SH}_n$, $\theta^H_n$, $\theta^{SH}_n$, $\theta^I_n$ as a function of $n$ in test case 2.}\label{Offline1Fex2}
\end{figure}

\begin{figure}[H]
\centering
   \includegraphics[width=12cm]{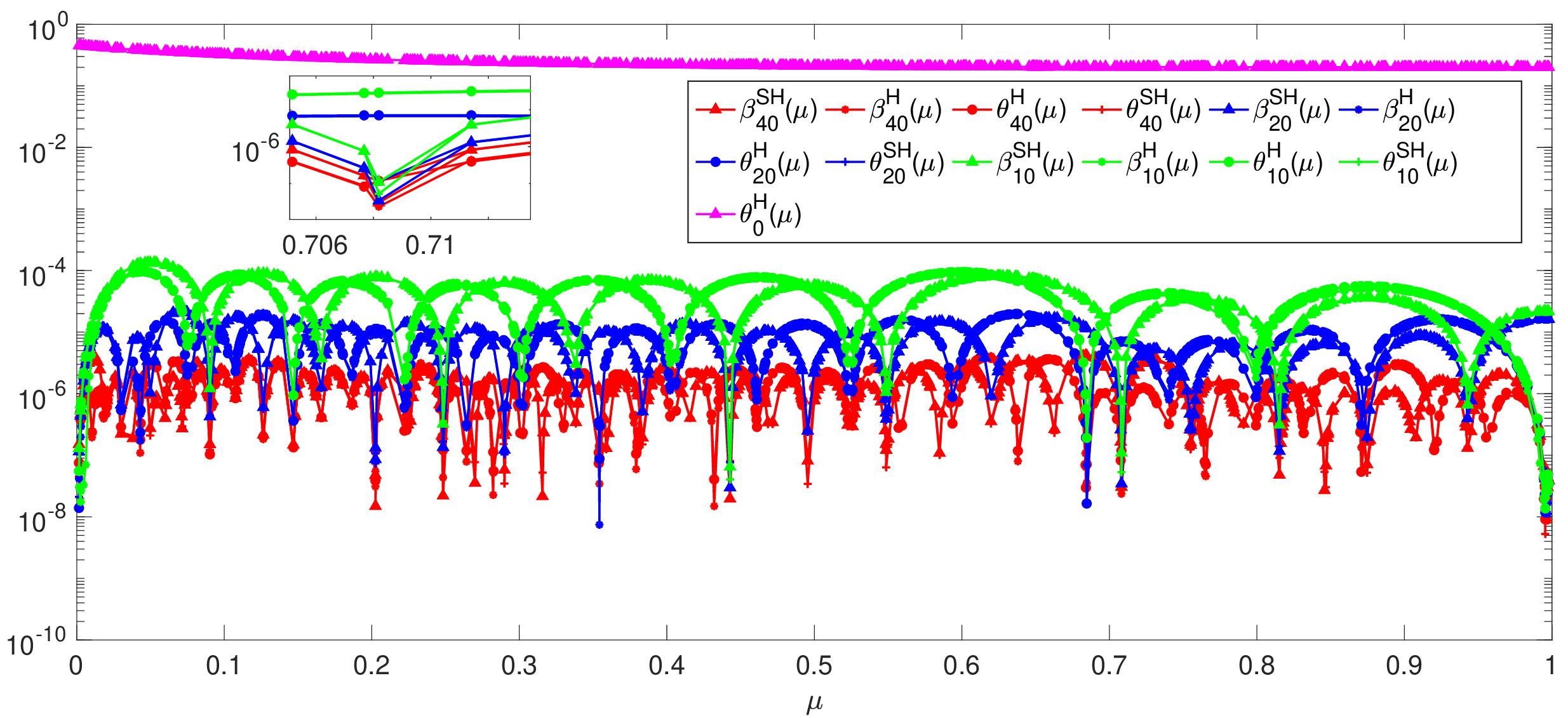}
     \caption{$\theta^H_n(\mu)$, $\theta_n^{SH}(\mu)$, $\beta_n^H(\mu)$, $\beta_n^{SH}(\mu)$ as a function of $\mu$ for $N=0,10,20,40$ on $\cP_{test}=[0,1]$.}\label{Online1Fex2}
\end{figure}

\begin{figure}[H]
\centering
   \includegraphics[width=12cm]{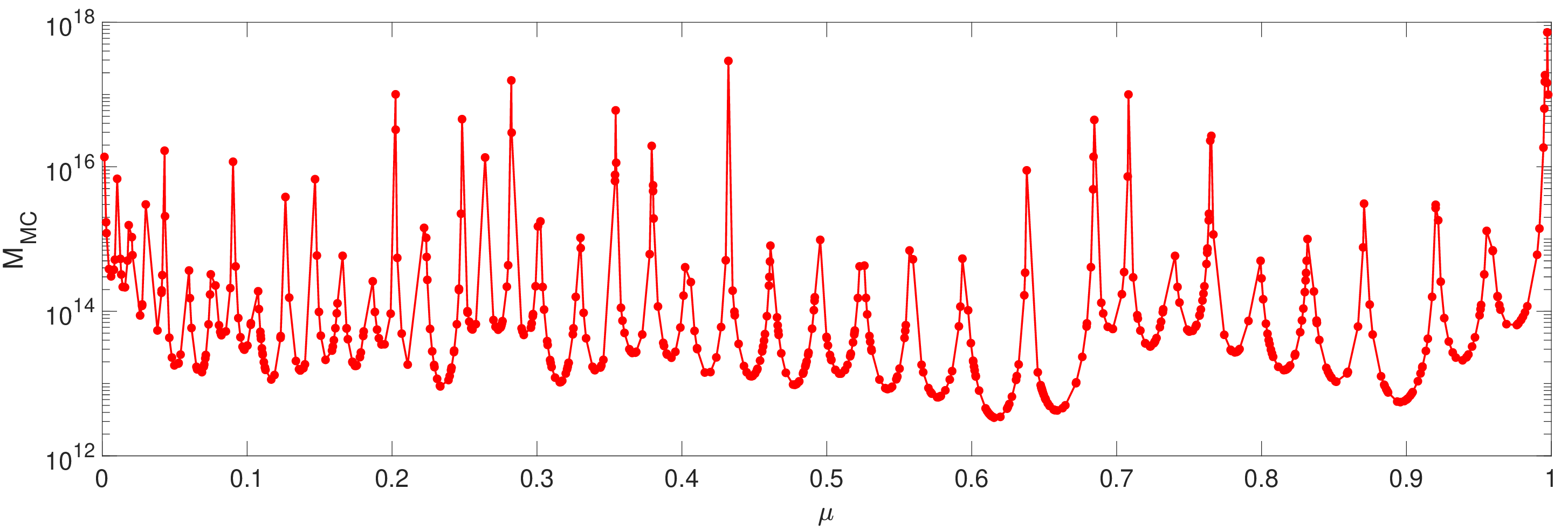}
     \caption{$M_{MC}(\mu)$ as a function of $\mu\in \cP_{test}=[0,1]$.}\label{MrefGainfex2}
\end{figure}
\medskip

%

\begin{figure}[H]
\centering
   \includegraphics[width=12cm]{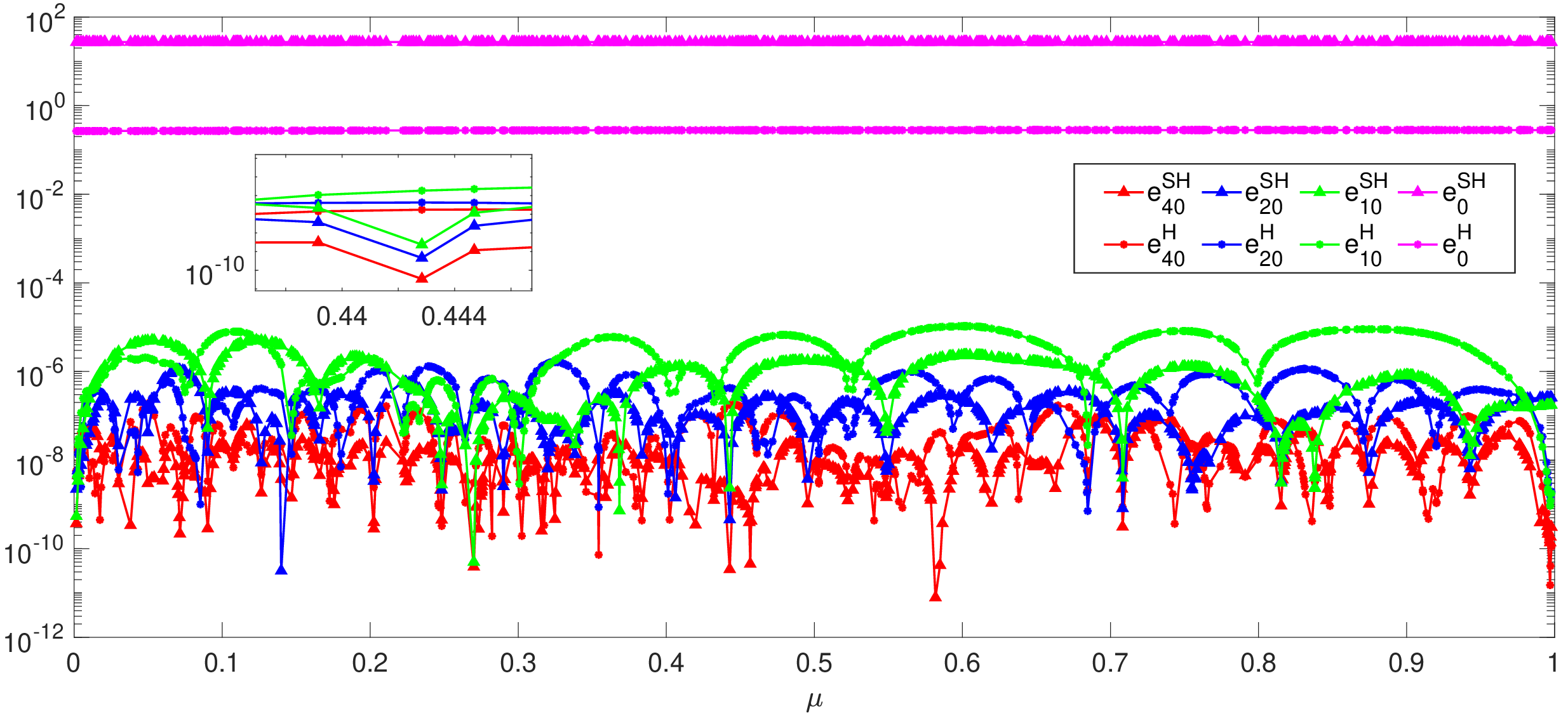}
     \caption{$e_N^H(\mu)$ and $e_N^{SH}(\mu)$ as a function of $\mu$ for $n=0,10,20,40$ on $\cP_{test}=[0,1]$.}\label{Online2Fex2}
\end{figure}

\subsection{Two-dimensional heat equation}

Let $Z_1$ and $Z_2$ be two independent real-valued random varibales with probability laws respectively $\mathcal{U}(0.5,2)$ and $\mathcal{N}(0,1))$ and let $Z =(Z_1, Z_2)$.
Let $\mathcal{D} = (0,2)^2$, $\mathcal P:=[0,10]$. The trial set $\mathcal P_{\rm trial}$ is constructed by selecting $50$ random values uniformly distributed in $\mathcal P$.

For all $\mu\in \mathcal P$ and $z:=(z_1,z_2)\in (0,5, 2) \times \mathbb{R}$, we introduce 
$$
D^{\mu,z}: \left\{ \begin{array}{ccc}
                \mathcal D  & \to & \mathbb{R}^{2\times 2}\\
                (x,y) & \mapsto & \begin{bmatrix}
 D^{\mu,z}_{11}(x,y) & 0 \\
0 & D^{\mu,z}_{22}(x,y)
\end{bmatrix} \\
\end{array}
\right.
$$
where 
$$
\forall (x,y)\in \mathcal D, \quad D_{11}^{\mu,z}(x,y) = 13+\mu \sin(2\pi x/z_1)+0.5 z_2 \quad \mbox{ and } D_{22}^{\mu,z}(x,y) = 13+\mu \sin(2\pi y/z_1)+0.5 z_2. 
$$

We introduce a conform triangular mesh $\mathcal T$ of the domain $\mathcal D$ as represented on the left-hand side plot of Figure~\ref{FigPDE} and denote by 
$$
V_h:= \left\{ u \in \mathcal{C}\left( \mathcal D \right), \quad u|_T \in \mathbb{P}_1 \; \forall T \in \mathcal T, \quad u|_{\partial \mathcal D} = 0 \right\}, 
$$
the standard $\mathbb{P}_1$ finite element space associated to this mesh.

For $\mu \in \cP$ and $z\in (0.5, 2)\times \mathbb{R}$, we define $u_h^{\mu, z}\in V_h$ the unique solution to
\begin{equation}\label{approx}
a_{\mu, z}\left( u_h^{\mu, z}, v\right) = b(v), \quad \forall v\in V_h,
\end{equation}
where 
$$
\forall v,w \in H^1_0(\mathcal D), \quad a_{\mu, z}  = \int_{\mathcal D} \nabla v \cdot D^{\mu, z} \nabla w, \quad b(v) = \int_{\mathcal D} r v,
$$
and where $r\in L^2(\mathcal D)$ is defined by
$$
r(x,y) = \exp{(-(x-1)^2-(y-1)^2)}, \quad \forall (x,y)\in \mathcal D.
$$
The function $u_h^{\mu, z}$ is thus the standard $\mathbb{P}_1$ finite element approximation of the unique solution $u^{\mu, z}\in H^1_0(\mathcal D)$ to 
\begin{equation}\label{PDE}
 \left\{
 \begin{array}{ll}
  - {\rm div}\left( D^{\mu,z} \nabla u^{\mu,z}\right) = r, & \mbox{ in } \mathcal D,\\
  u^{\mu,z} = 0 & \mbox{ on } \partial \mathcal D.\\
 \end{array}
\right.
\end{equation}

Let $T_1\in \mathcal T$ be the triangle colored in red in the left-hand side plot of Figure~\ref{FigPDE}. For all $\mu \in \mathcal P$ and $z\in (0,5,2) \times \mathbb{R}$, we define by
$$
f_\mu(z):= \frac{1}{|T_1|}\int_{T_1}u^{\mu,z}_h.
$$

\medskip

 \begin{figure}[h]
   \begin{minipage}{0.3\textwidth}
    \includegraphics[scale=0.3]{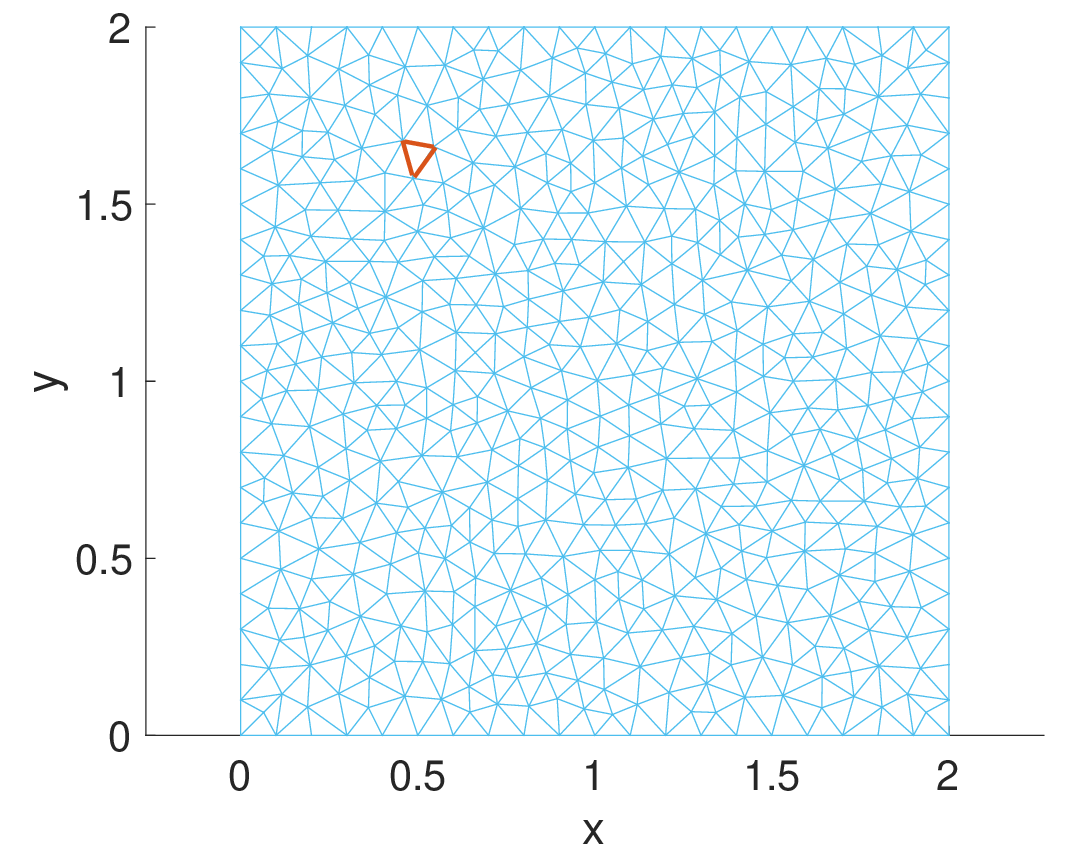}
 \end{minipage}\hspace{5mm}
 \begin{minipage}{.3\textwidth}
     \includegraphics[scale=0.3]{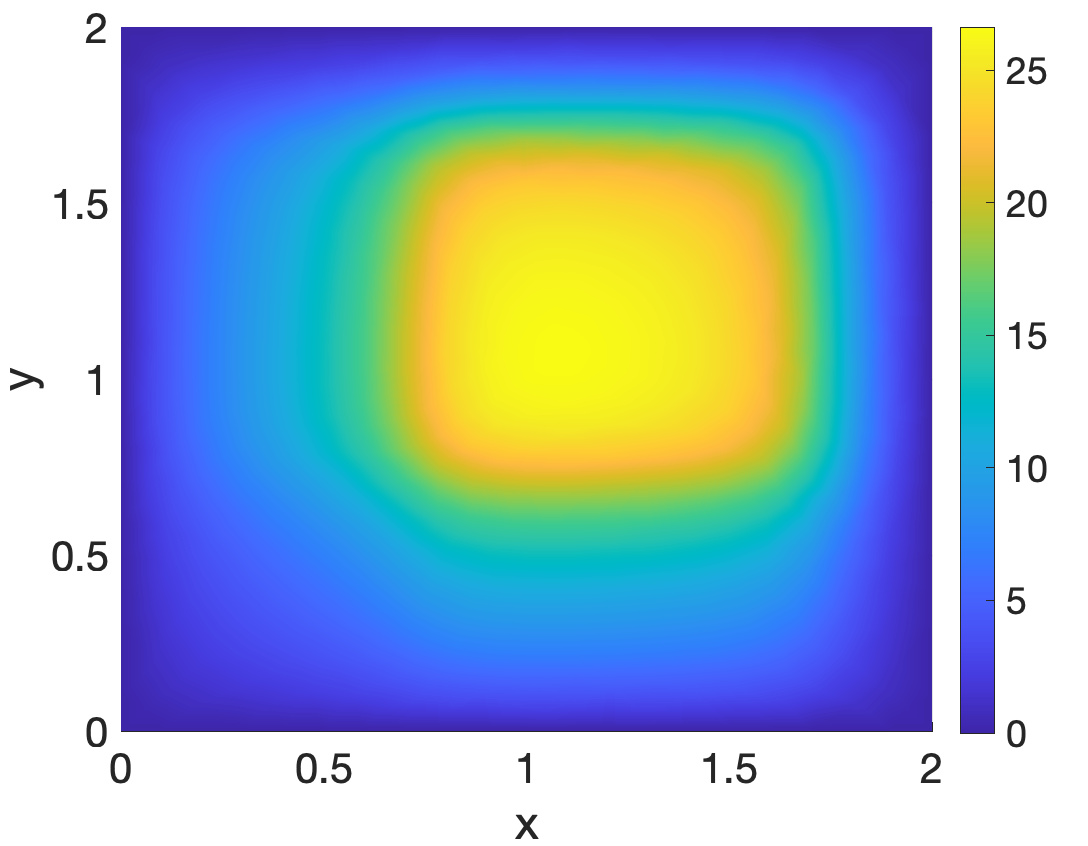}
    \end{minipage}\hspace{5mm}
    \begin{minipage}{0.3\textwidth}
    \includegraphics[scale=0.3]{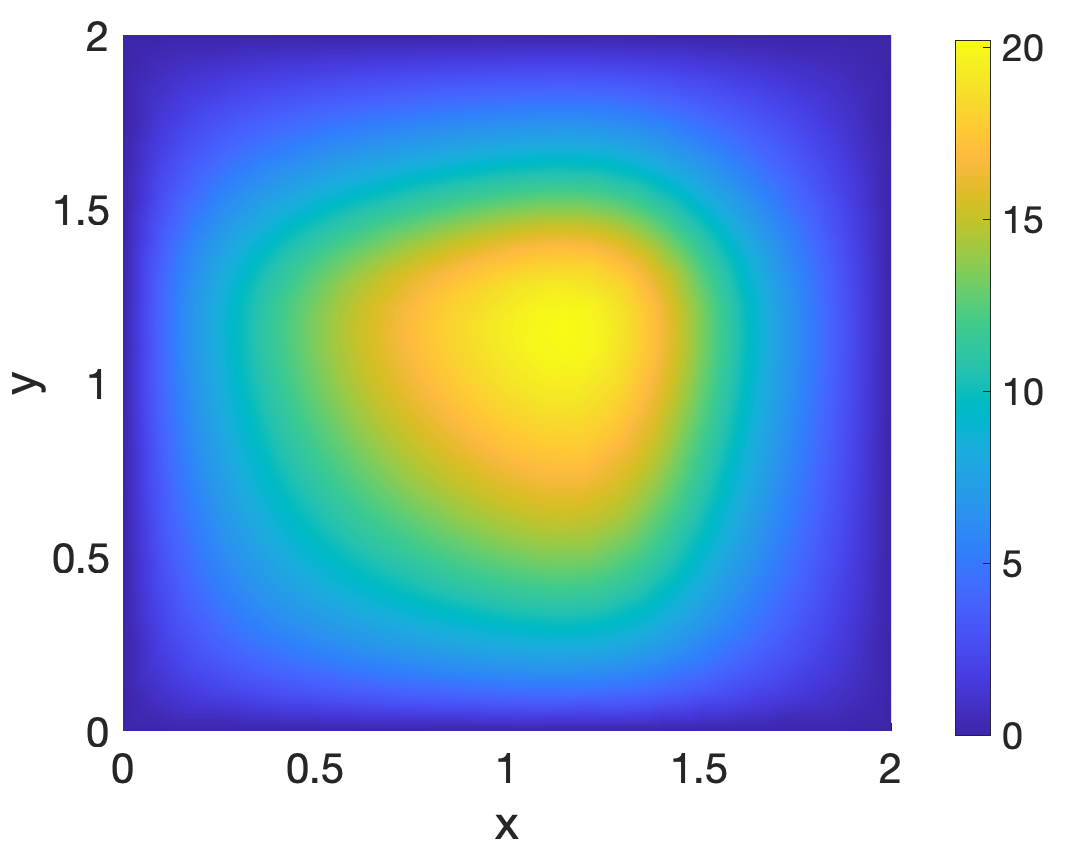}
 \end{minipage}
    \caption
    {Left: mesh $\mathcal T$ (the triangle $T_1$ is highlighted in red color); Center: $u^{\mu, z}_h$ for $\mu=9$ and $z=(1,0)$; Right: $u^{\mu, z}_h$ for $\mu = 9$ and $z = (1.777,0.2062)$.} \label{FigPDE}
\end{figure}

\medskip

In this example, $M_{\rm ref} = 10^5$, $M_1 = 800$ and $\gamma = 0.9$. Figure~\ref{fig:figM_Dir} illustrates the evolution of the values of $M_n$ as a function of $n$ for the HMC and SHMC algorithms. 

 \begin{figure}[h]
 \begin{center}
    \includegraphics[width=12cm]{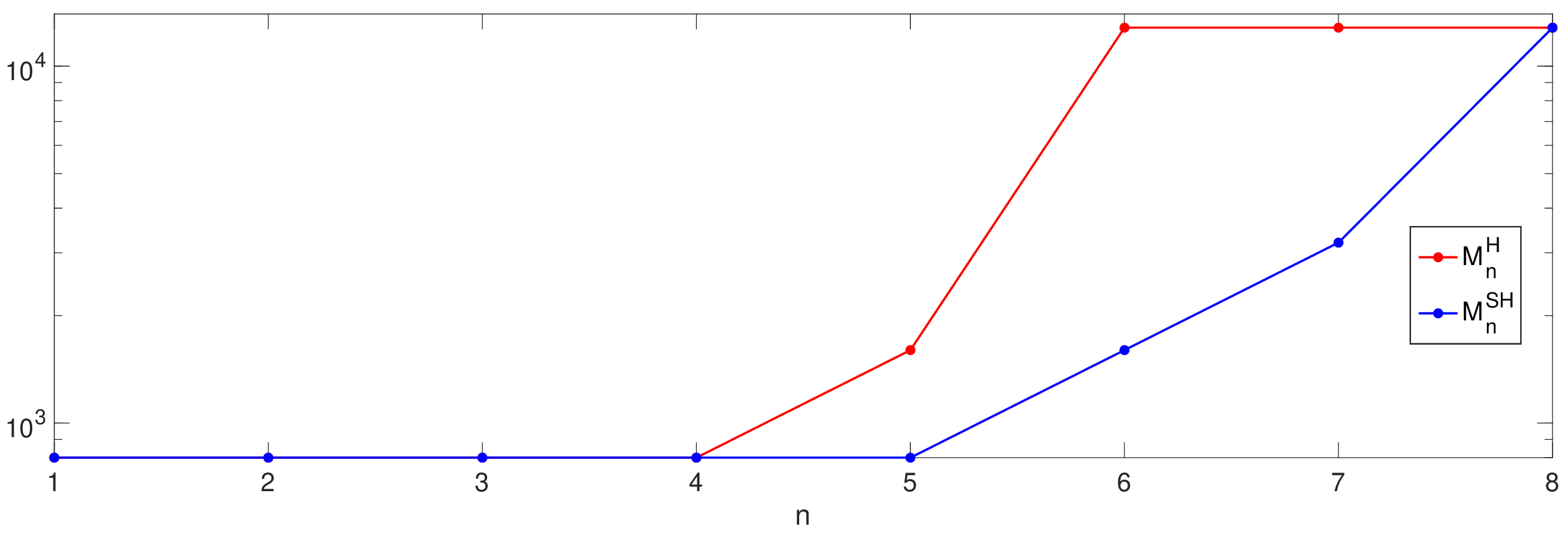}
    \caption
    {Evolution of $M_n$ as a function of $n$ for the HMC and SHMC-greedy algorithms in test case 3.} \label{fig:figM_Dir}
\end{center}
    \end{figure}

\medskip

\normalfont

It is to be noted here that the quantities $\theta^H_n$, $\theta^{SH}_n$ and $\theta_n^I$ are very close: the quality of approximation of the reduced spaces $V_n^H$ or $V_n^{SH}$ is very close to the quality of approximation of 
the reduced space $V_n^I$ given by an ideal greedy algorithm.

\medskip

\begin{figure}[H]
\centering
   \includegraphics*[width = 12cm]{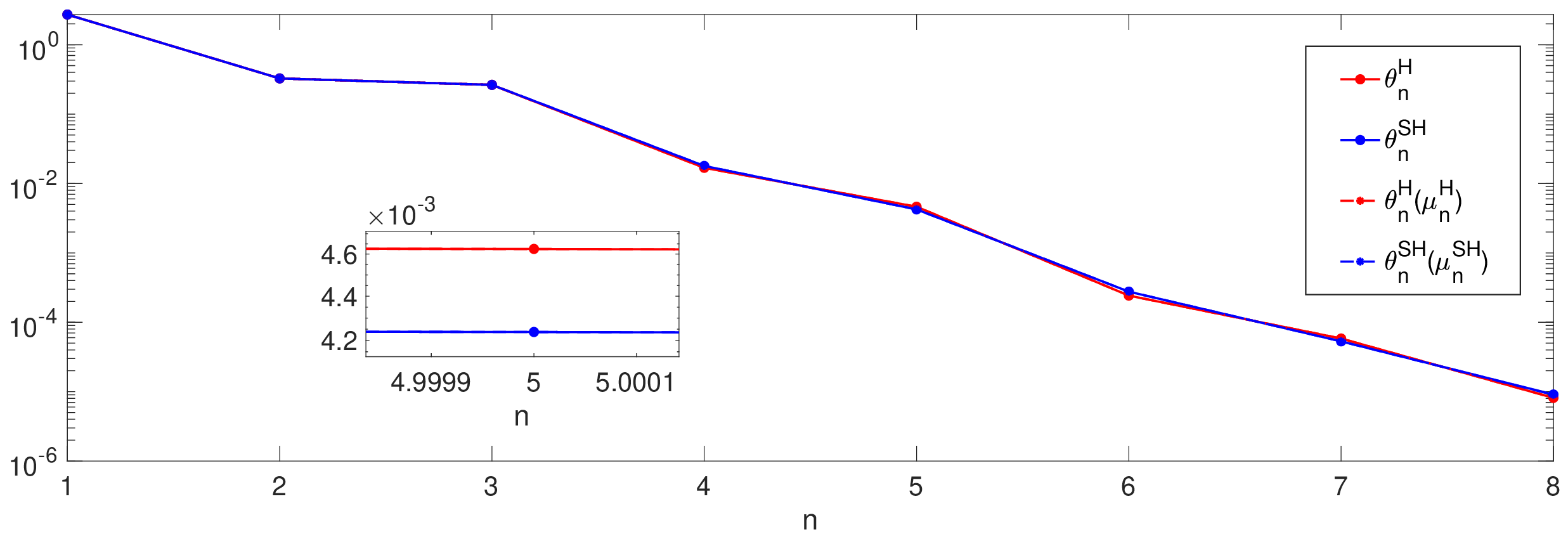}
     \caption{Evolution of $\theta^H_n(\mu_n^H)$, $\theta^{SH}_n(\mu_n^{SH})$, $\theta^H_n$, $\theta^{SH}_n$as a function of $n$ in test case 3.}\label{ErrMaxD1}
\end{figure}

\begin{figure}[H]
\centering
   \includegraphics*[width =12cm]{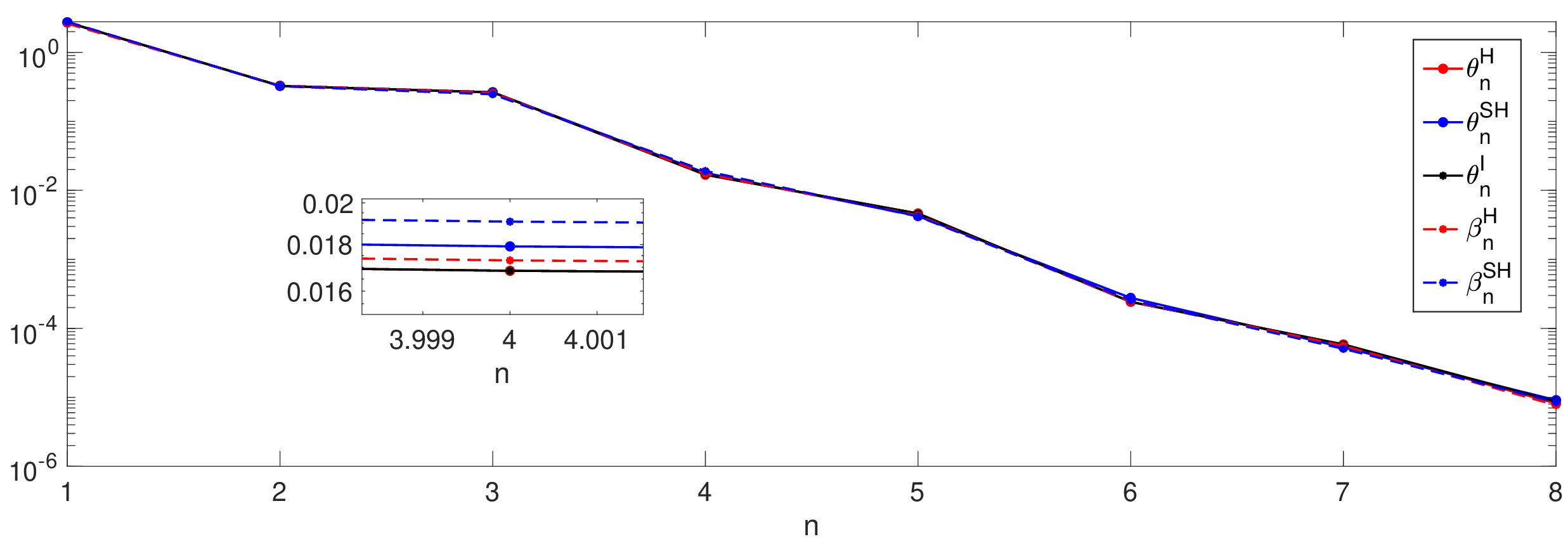}
     \caption{Evolution of $\beta^H_n$, $\beta^{SH}_n$, $\theta^H_n$, $\theta^{SH}_n$, $\theta^I_n$ as a function of $n$ in test case 3.}\label{ErrMaxD2}
\end{figure}

\begin{figure}[H]
\centering
   \includegraphics[width=12cm]{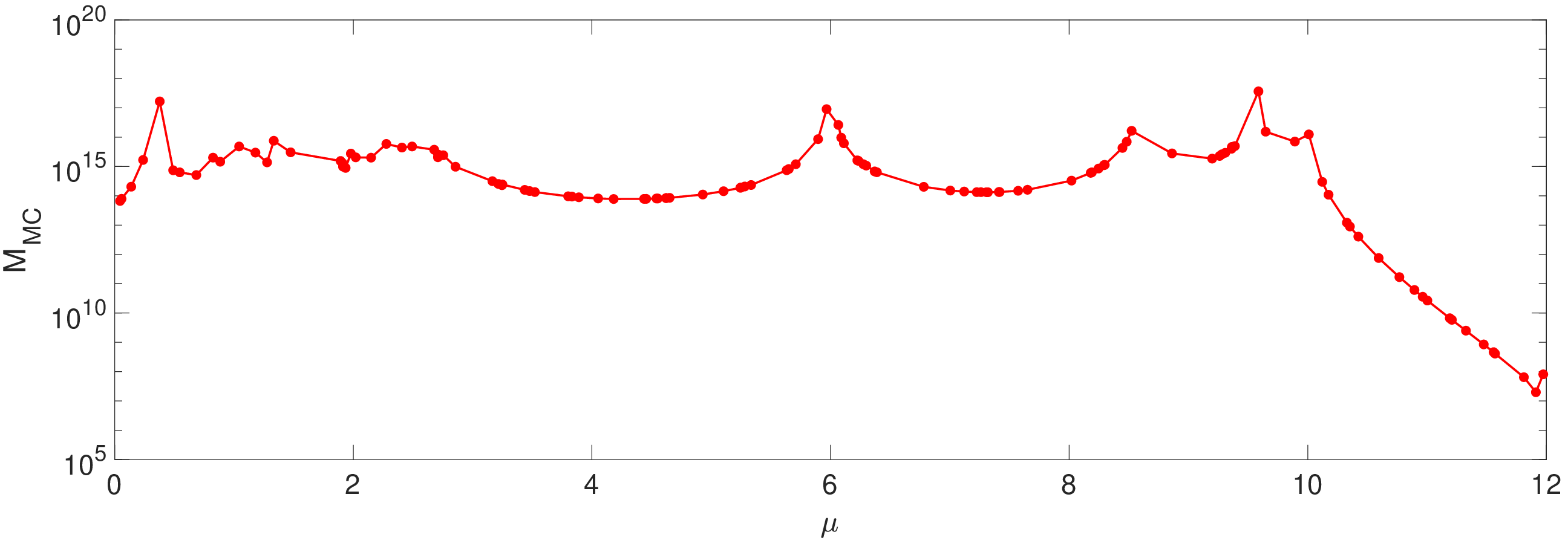}
     \caption{$M_{MC}(\mu)$ as a function of $\mu\in\cP_{test}=[0,12]$.}\label{ErrMnIn0}
\end{figure}

Figure~\ref{ErrMnIn0} shows the value of $M_{MC}(\mu)$ given by (\ref{eq:MMC}), knowing that $M_n = 12800$ after $n=7$ iterations of the HMC algorithm. 
We observe that in this case $10^{14}\leq M_{MC}(\mu)\leq 10^{20}$, which shows the huge computational gain brought by the HMC algorithm with respect to a standard Monte-Carlo method in this test case. 

\medskip

\begin{figure}[H]
\centering
   \includegraphics[width=12cm]{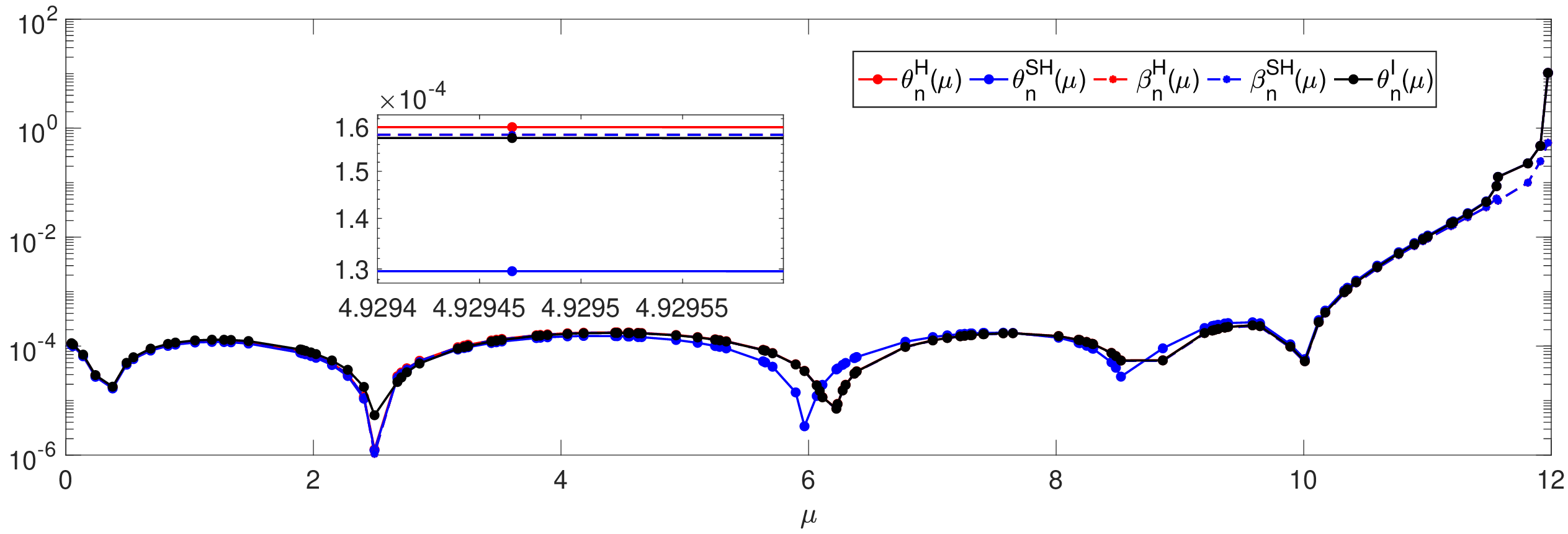}
     \caption{$\theta^H_n(\mu)$, $\theta_n^{SH}(\mu)$, $\theta_n^I(\mu)$, $\beta_n^H(\mu)$, $\beta_n^{SH}(\mu)$ as a function of $\mu$ for $n=5$ in test case 3.}\label{ErrMnIn1}
\end{figure}

\medskip

%

\begin{figure}[H]
\centering
   \includegraphics[width=12cm]{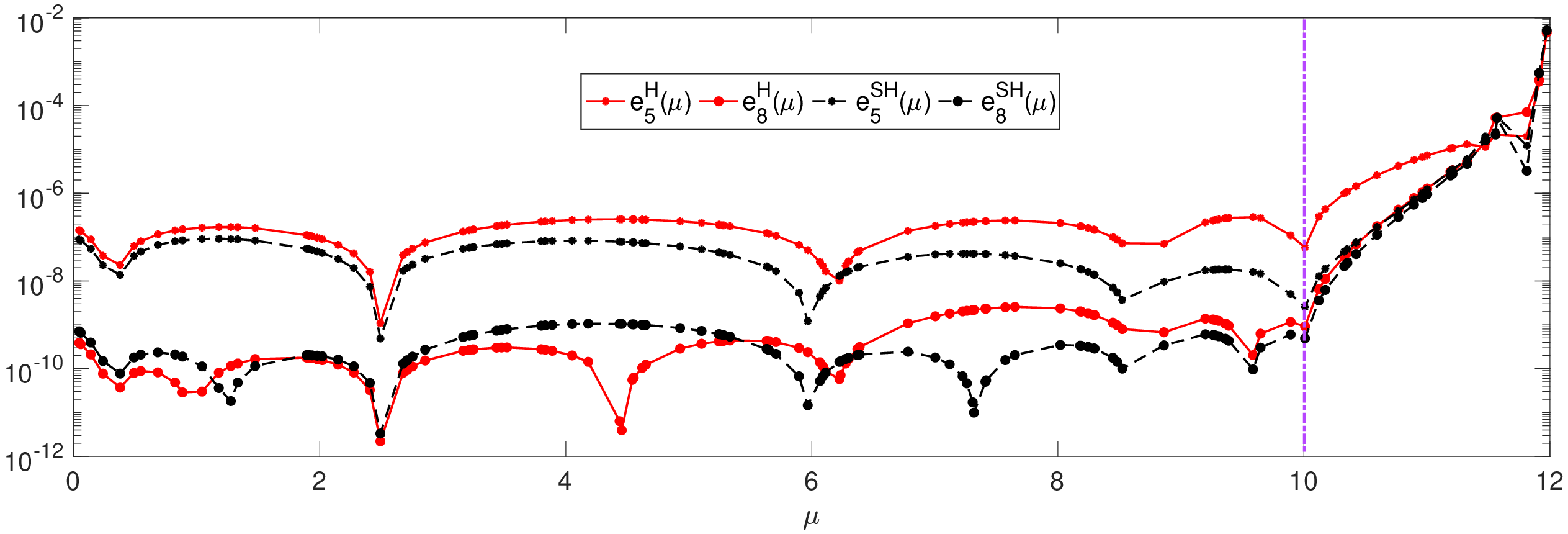}
     \caption{$e_n^H(\mu)$ and $e_n^{SH}(\mu)$ as a function of $\mu$ for $n=5$ and $n=8$ in test case 3.}\label{ErrEML}
\end{figure}

\subsection*{Acknowledgements}

The authors acknowledges financial support from the Mohammed VI Polytechnic University for funding Raed Blel's PhD thesis. Virginie Ehrlacher acknowledges support from project COMODO (ANR-19-CE46-0002). This work is partially supported by a funding from the European
ResearchCouncil (ERC) under the European Union’s Horizon 2020 research
and innovation programme (grant agreement No 810367), project EMC2.

\bibliographystyle{plain}
\bibliography{bibliography}

\end{document}